\documentclass[10pt]{amsart}

\usepackage{amssymb,amsthm,amsfonts,latexsym}
\usepackage{amsmath}
\usepackage{mathrsfs}
\usepackage{stmaryrd}
\usepackage{accents}
\usepackage[utf8]{inputenc}
\usepackage{color}
\usepackage{enumitem}
\usepackage{tabularx}
\usepackage{hyperref}

\input{xy}
\xyoption{all}
\xyoption{poly}
\usepackage[all]{xy}
\allowdisplaybreaks

\theoremstyle{plain}
\newtheorem{theorem}{Theorem}[section]
\newtheorem{lemma}[theorem]{Lemma}
\newtheorem{corollary}[theorem]{Corollary}
\newtheorem{proposition}[theorem]{Proposition}

\newtheorem*{theorem*}{Theorem}
\newtheorem*{corollary*}{Corollary}

\theoremstyle{definition}
\newtheorem{example}[theorem]{Example}
\newtheorem{definition}[theorem]{Definition}

\newtheorem{remark}[theorem]{Remark}

\newcommand{\donerk}{\hfill $\diamondsuit$ \smallskip }

\numberwithin{equation}{section}

 \renewcommand\AA{{\mathbb{A}}}
 
 \newcommand\QQ{\mathbb{Q}}
 \renewcommand\SS{{\mathbb{S}}}

 \def\A{{\mathcal A}}
 \def\B{{\mathcal B}}
 
 \def\D{{\mathcal D}}

 \def\K{{\mathcal K}}

 \def\N{{\mathcal N}}
 \def\O{{\mathcal O}}
 \def\P{{\mathcal P}}

 \def\S{{\mathcal S}}

 \DeclareMathOperator{\HS}{HS}

 \DeclareMathOperator{\Sz}{Sz} 
 \DeclareMathOperator{\Ree}{Ree}
 
 \DeclareMathOperator{\Aut}{Aut}
 
 \DeclareMathOperator{\Out}{Out}
 \DeclareMathOperator{\Inn}{Inn}

 \DeclareMathOperator{\Outposet}{I}
 \newcommand{\Imageposet}[2]{{\mathcal{A}_{#1, #2}}}
 
 \DeclareMathOperator{\Syl}{Syl}

 \DeclareMathOperator{\Id}{Id}
  
 \DeclareMathOperator{\Lk}{Lk}
 \DeclareMathOperator{\St}{St}
 
 \def\join{*}

\newcommand{\normal}{\trianglelefteq}
\def\groupiso{\cong}

\newcommand{\tq}{\mathrel{{\ensuremath{\: : \: }}}}

\def\Im{\mathrm{Im}}

\newcommand\gen[1]{\left\langle#1\right\rangle}

      \makeatletter
      \def\@setcopyright{}
      \def\serieslogo@{}
      \makeatother
   
\begin{document}

\title [Eliminating components in Quillen's Conjecture]
       {Eliminating components in Quillen's Conjecture}

%\date{ drafted~9nov2020; PRELIMINARY version~segev.6.sds as of 24jan2021 }

\author{ Kevin Iv\'{a}n Piterman* \\
          Departamento de Matem\'{a}tica \\
	  IMAS-CONICET, FCEyN \\
	  Universidad de Buenos Aires \\
	  Buenos Aires  ARGENTINA \\  
	  e-mail: { \tt kpiterman@dm.uba.ar } \\ 
	       \\ 
        Stephen D. Smith \\ 
         Department of Mathematics  \\
         University of Illinois at Chicago \\
         Chicago Illinois USA \\
	 (home: 728 Wisconsin, Oak Park IL 60304 USA) \\
	 e-mail: {\tt smiths@math.uic.edu}       \\
       }

\thanks{*Supported by a CONICET postdoctoral fellowship and grants PIP 11220170100357, PICT 2017-2997, and UBACYT 20020160100081BA}

% abstract 
\begin{abstract}
We generalize an earlier result of Segev,
which shows that {\em some\/} component in a minimal counterexample to Quillen's conjecture must admit an outer automorphism.
We show in fact that {\em every\/} component must admit an outer automorphism.
Thus we transform his restriction-result on components to an elimination-result:
namely one which excludes any component which does not admit an outer automorphism.
Indeed we show that the outer automorphisms admitted must include~$p$-outers:
that is, outer automorphisms of order divisible by~$p$.
This gives stronger, concrete eliminations:
for example if~$p$ is odd, it eliminates sporadic and alternating components---thus
reducing to Lie-type components (and typically forcing~$p$-outers of field type).
For~$p = 2$, we obtain similar but less restrictive results.
We also provide some tools to help eliminate suitable components that do admit~$p$-outers in a minimal counterexample.

% on the kernel on components, new tools and results towards Quillen's conjecture.
%Segev's theorem establishes Quillen's conjecture under suitable conditions in all the components of the group.
%In this article, we focus on the behavior of a single component $L$ of a group $G$ to establish Quillen's conjecture,
%by studying the map $\A_p(L)\to \A_p(\Aut_G(L))$ and the centralizers of outer automorphisms of $L$.
%Recall that $\A_p(G)$ is the poset of nontrivial elementary abelian $p$-subgroups of $G$.
%We provide useful criteria to decide when we can propagate homology from $\A_p(L)$ to $\A_p(G)$ if $L$ is a component of $G$,
%by looking into the centralizers in $L$.
%As a corollary, we show that in a group $G$ of minimal order subject to failing Quillen's conjecture,
%every component $L\leq G$ must admit a nontrivial elementary abelian $p$-subgroup $E\leq G$ inducing outer automorphism on $L$.
%In particular,
%it shows that if $p$ is odd then a minimal order counterexample to the conjecture has neither alternating nor sporadic components.
%For $p = 2$ we obtain similar, but less restrictive, conclusions on the structure of the components.
\end{abstract}



\maketitle

\tableofcontents

\centerline{
    THIS PAPER IS DEDICATED TO THE MEMORY OF JAN SAXL
            }

\section{Introduction}

This paper began life as a sequel to Segev's article
``Quillen's conjecture and the kernel on components"~\cite{Segev}; and expanded from there.
In brief summary:

\bigskip

One consequence of Segev's main result---essentially the contrapositive
of the final remark on~$\Out_G(L)$ in Theorem~\ref{theoremSegevOriginal} below---can be stated in the form:

\begin{corollary}[{Segev}]
\label{corollarySegev}
If~$G$ is a minimal counterexample to Quillen's conjecture, then~$G$ induces outer automorphisms on {\em some\/} component.
\end{corollary}

\noindent
We call this kind of result on a counterexample to Quillen's conjecture a \textit{restriction} result,
since it restricts at least one component to have a certain Property~$P_G$. 
(Here, $P_G$ would be that the component admits an outer automorphism induced in~$G$.)

By contrast, a consequence of our extension of Segev's work can be stated in the related form:

\begin{corollary}
\label{corollaryMain}
If~$G$ is a minimal counterexample to Quillen's conjecture, then~$G$ induces outer automorphisms on {\em all\/} components.
\end{corollary}

\noindent
We call this kind of result an \textit{elimination} result, since it establishes
that every component in a minimal counterexample to Quillen's conjecture must have a certain Property~$P_G$---so that
a component with not-$P_G$ (here, one having no outers in~$G$) is eliminated from a minimal counterexample.

In fact, we can sharpen Corollary~\ref{corollaryMain}, to eliminate components with only $p'$-outers;
the following is essentially the contrapositive of the final statement about~$\Out_G(L)$ in Theorem~\ref{theorem2} below:

\begin{corollary}
\label{corollaryPOuter}
If~$G$ is a minimal counterexample to Quillen's conjecture,
then~$G$ induces outer automorphisms of order~$p$ on {\em all\/} components.
\end{corollary}

\noindent
As a preview of our elimination-applications:
Assume $p$ is odd.
Then Corollary~\ref{corollaryPOuter} eliminates alternating and sporadic components~$L$, since they have~$\Out(L)$ a~$2$-group.
So by the Classification of Finite Simple Groups~(CFSG), all the components are of Lie-type.
These have~$\Out(L)$ given by diagonal$\backslash$field$\backslash$graph automorphisms.
But we can have diagonal$\backslash$graph automorphisms for only some~$p$; so ``mostly" we get just field automorphisms.

\bigskip
\noindent
In the remainder of this Introduction, we further expand on the above summary: 

\bigskip

We begin with some background on Quillen's Conjecture:
Let~$G$ be a finite group and~$p$ a prime dividing its order.
In~\cite{Qui78},
Quillen introduced the poset~$\A_p(G)$ of nontrivial elementary abelian~$p$-subgroups
and studied its homotopical and topological properties  via its order complex.
He showed that if~$G$ has a nontrivial normal~$p$-subgroup, then~$\A_p(G)$ is contractible;
and conjectured the converse: 
that if~$\A_p(G)$ is contractible, then~$G$ should have a nontrivial normal~$p$-subgroup.
Contrapositively, if the largest normal~$p$-subgroup of~$G$ is trivial, then~$\A_p(G)$ should not be contractible.
This is the well-known \textit{Quillen Conjecture}, which we abbreviate by~(QC).
The Conjecture remains open in general;
but there have been important advances, such as~\cite{AK90, AS93, KP20, Qui78, Segev}
(see~\cite[Ch.8]{Smi11} for a fuller historical discussion).

In this article, we restrict attention to the following stronger homology-version of~(QC);
recall that~$O_p(G)$ denotes the largest normal~$p$-subgroup of~$G$:

\vspace{0.2cm}
\begin{tabular}{cc}
(H-QC) & If $O_p(G) = 1$, then $\tilde{H}_*(\A_p(G),\QQ)\neq 0$.
\end{tabular}
\vspace{0.2cm}

\noindent
Here, $\tilde{H}_*(X,\QQ)$ denotes the rational homology of the finite poset~$X$, which is the homology of its order complex.
We will work with rational homology,
so in general we will drop the coefficient notation and write~$\tilde{H}_*(X)$ for~$\tilde{H}_*(X,\QQ)$.

We next indicate some existing~(H-QC)-results, as background to Segev's work:
Quillen established~(H-QC) for solvable groups, and groups of~$p$-rank at most~$2$.
Later, various authors extended~(H-QC) to~$p$-solvable groups (see~\cite[8.2.12]{Smi11}).
In this direction, in~\cite[1.6]{AS93} and~\cite{KP20} (cf.~Theorem~\ref{generalReduction} below),
it is proved that a minimal counterexample~$G$ to~(H-QC) has~$O_{p'}(G) =1$.%
\footnote{
   Recall that~$O_{p'}(G)$ denotes the largest normal~$p'$-subgroup of~$G$.
          }
It then follows that the generalized Fitting subgroup~$F^*(G)$ is the direct product of the components of~$G$---which
are simple,%
\footnote{
  Our convention is that that ``simple'' means  {\em non-abelian\/} simple.
          }
and of order divisible by~$p$ (see Lemma~\ref{lemmaOpandp} below).
The case where~$G$ itself is simple was included in the work of Aschbacher and Kleidman in~\cite{AK90}---who
established~(H-QC) for almost-simple groups (i.e., where~$F^*(G)$ is simple).
Subsequently Aschbacher and Smith in~\cite{AS93} obtained a result for fairly general~$G$---showing
that~(H-QC) holds for~$p > 5$ under suitable constraints on the components of~$G$ which are unitary groups.
Along the way, in~\cite[1.7]{AS93} it is shown
that every component in a minimal counterexample to~(H-QC) must fail the~QD-property
(``Quillen dimension'', defined at~\cite[p.474]{AS93})---notice
this is an elimination result, ruling out components with~QD.

\bigskip

In~\cite{Segev}, Segev worked with the kernel on components of~$G$.
This subgroup of~$G$ is the kernel of the conjugation action of~$G$ on its set of components, namely:
  \[ H := \bigcap_L N_G(L) . \]
Here the intersection runs over all the components~$L$ of~$G$, and~$N_G(L)$ denotes the normalizer of~$L$ in~$G$.
Segev established~(H-QC) under certain conditions on the kernel~$H$ on components, when~$O_{p'}(G) = 1$.
As indicated in the previous paragraph, $F^*(G)$ is the product of the components. 
In particular, his main result (stated as Theorem~\ref{theoremSegevOriginal} below)
gives~(H-QC) under the hypothesis that~$H = F^*(G)$---which
implies that the kernel~$H$ on components induces only {\em inner\/} automorphisms on each component:

\begin{theorem}[{Segev, \cite[Thms~2 \& 3]{Segev}}]
\label{theoremSegevOriginal}
Suppose that~$O_{p'}(G) = 1$.
Let~$H := \bigcap_{L} N_G(L)$ be the kernel on components.
If for each component~$L$ of~$G$, the inclusion map~$\A_p(L) \subseteq \A_p(\Aut_H(L))$ is not the zero map in homology,
then~$G$ satisfies~(H-QC).

In particular, if~$H=F^*(G)$, then~$G$ satisfies~(H-QC).

Indeed if~$\Out_H(L)$
or~$\Out_G(L)$ is a~$p'$-group for every component~$L$ of~$G$, then~$G$ has~(H-QC).
\end{theorem}

\noindent
Here $\Aut_G(L) = N_G(L)/C_G(L)$, with~$C_G(L)$ the centralizer of~$L$ in~$G$;
also~$\Out_G(L) = \Aut_G(L)/L$.
%The equality $H = F^*(G)$ means that $\Aut_H(L) = L$ for every component $L$ of $G$.
We have~$\Aut_H(L) \leq \Aut_G(L)$ and~$\Out_H(L) \leq \Out_G(L)$.

\bigskip 

We will next indicate some ways in which we can weaken Segev's hypotheses.
Note that Segev's theorem requires a common behavior in all the components of~$G$; and also requires~$O_{p'}(G) = 1$.

In this article, we will show that we can instead focus on the behavior of a single component.
Therefore, we will convert Segev's restriction result into an elimination result---as
we had indicated earlier, when we contrasted Corollaries~\ref{corollarySegev} and~\ref{corollaryMain}.

We will also drop the requirement that~$O_{p'}(G) = 1$.
Instead, we will often work under one of the following inductive hypotheses~(H1) or~(H$L(p)$)---which
will at least limit some aspects of behavior of the components.
Hypothesis~(H1) below is motivated by previous results~\cite{AS93,KP20}, with an eye to a minimal counterexample to~(H-QC):%
\footnote{
   Recall that a proper central quotient of~$G$ is a quotient by a nontrivial subgroup~$Z \leq Z(G)$,
   where~$Z(G)$ is the center of~$G$.
          }

\vspace{0.2cm}

\begin{tabular}{cc}
(H1) & Proper subgroups and proper central quotients of~$G$ satisfy~(H-QC).
\end{tabular}

\vspace{0.2cm}

\noindent
And indeed in the context of a counterexample of minimal order to~(H-QC), we do have the above hypothesis~(H1).
Hence, results proved under~(H1) will also hold under the minimal-counterexample variant of~(H1):
namely that~(H-QC) holds for all groups~$H$ such that~$|H|<|G|$.
For our other inductive hypothesis, let~$L$ be a component of~$G$, and~$L_1,\ldots,L_t$ its~$G$-orbit; define:

\vspace{0.2cm}

\begin{tabular}{cc}
(H$L(p)$) & $p$ divides the order of~$L$, and~$C_G(L_1 \ldots L_t)$ satisfies~(H-QC).
\end{tabular}

\vspace{0.2cm}

\noindent
In fact, later Remark~\ref{remarkH1andHLp} gives an implication-relation between our two hypotheses:
Namely under~(H1), $G$ satisfies either~(H-QC), or~(H$L(p)$) for every component~$L$ of~$G$.
So within the context of proving~(H-QC), the hypothesis~(H$L(p)$) is weaker than~(H1).
Hence in proving various results,
e.g.~our main~(H-QC)-result Theorem~\ref{theorem2} below (and Corollary~\ref{cor:genTheorem2Segev} which leads to it), 
we will be able to assume as hypothesis ``either~(H1), or~(H$L(p)$) for some~$L$''---and then it suffices 
to give the proof just under~(H$L(p)$) for that~$L$. 

\bigskip

Now we turn our attention away from hypotheses, towards sharpening Segev's conclusions---that is, to obtaining elimination-results.
Namely in Theorem~\ref{theorem2} below, we will see that, under~(H1), 
if there is some component~$L$ of~$G$ such that no elementary abelian~$p$-subgroup of~$G$ induces outer automorphisms on~$L$,
then~$G$ satisfies~(H-QC).

We then immediately obtain many of our particular elimination-results as consequences:
Recall we had mentioned that Corollary~\ref{corollaryPOuter} (and hence also Corollary~\ref{corollaryMain}) 
follows via the contrapositive of the final remark on~$\Out_G(L)$ in Theorem~\ref{theorem2}.

First assume~$p$ is odd.
Then using Corollary~\ref{corollaryPOuter},
we can eliminate alternating and sporadic components from a minimal counterexample to~(H-QC):
For (using e.g.~\cite[Secs~5.2,5.3]{GLS98}), if~$L = \AA_n$ is the alternating group on~$n$ letters,
then~$\Out(L) = C_2$ for~$n \neq 6$, and $\Out(\AA_6) = C_2\times C_2$;
while if~$L$ is a sporadic group, then~$\Out(L)\leq C_2$.
Therefore, if~$p$ is odd and~$L$ is a component of~$G$ of alternating or sporadic type, $\Out_G(L)$ is a~$p'$-group.
(This gives case~(1) below.)

Now assume $p=2$. 
Using Corollary~\ref{corollaryPOuter} as above,
we can also eliminate Suzuki (including~${^2}F_4$) and Ree components:
since their outer automorphisms are~$2'$-groups (using for example~\cite[Ch~4]{GLS98}).
(This gives case~(4) below.)

Furthermore we will also able, in the final two sections of the paper, to obtain elimination-consequences from various other results
(such as Proposition~\ref{prop:TrivialInclusionCentralizer}), which give us some tools
toward suitable conditions on the outer automorphisms of a simple group (and on the fixed points of those automorphisms)---in order
to then guarantee the hypotheses of Theorem~\ref{theorem2}, and hence establish~(H-QC) for~$G$.
Such results allow us to also eliminate, for example,
at least some alternating and sporadic components (see cases~(2) and~(3) below)
from a minimal counterexample to~(H-QC), in the present subcase~$p = 2$ as well.

We summarize our various concrete elimination-consequences of the above types in:

\begin{corollary}\label{corollaryComponents}
Assume~(H1).  Then~(H-QC) holds if~$p$ and~a component~$L$ of~$G$ satisfy one of:
\begin{enumerate}
\item $p$ is odd, and~$L$ is alternating or sporadic; or:
\end{enumerate}

We have~$p = 2$, with~$L$ given by one of:

\begin{enumerate}
\item[(2)] $L$ is sporadic~$\HS$.
\item[(3)] $L$ is alternating~$\AA_6$ or~$\AA_8$.
\item[(4)] $L$ has one of the Lie types~$\Sz(q)$, or~$^2F_4(q)$, or~$\Ree(q)$.
\end{enumerate}
\end{corollary}

\noindent
Note furthermore that any simple~$L$ with~$\Out(L) = 1$
is similarly eliminated (for any~$p$) by Corollary~\ref{corollaryMain}. 
For example, the reader can consult~\cite[Table~5.3]{GLS98} for the~$14$ sporadic~$L$ with~$\Out(L) = 1$.

\bigskip

We conclude our Introduction with the promised statement of our main general theorem leading to elimination results:

\bigskip

\begin{theorem}
\label{theorem2}
Let~$L$ be a component of~$G$ and~$L_1,\ldots, L_t$ its~$G$-orbit; and set~$H := \bigcap_i N_G(L_i)$.
Suppose that~$G$ satisfies~(H1) or~(H$L(p)$).

If~$\A_p(L) \subseteq \A_p \bigl( \Aut_H(L) \bigr)$ is not the zero map in homology, then~$G$ satisfies~(H-QC).

In particular, if we have~$H = L_1\ldots L_tC_G(L_1\ldots L_t)$, then~$G$ satisfies~(H-QC).

Indeed, if~$\Out_H(L)$~or $\Out_G(L)$ is a~$p'$-group, then~(H-QC) holds.
\end{theorem}

\noindent
Note that the ``In particular" part of Theorem~\ref{theorem2} applies for {\em some\/} component of $G$
(in contrast to {\em each\/} component, in Segev's Theorem \ref{theoremSegevOriginal});
so that its contrapositive forces nontrivial~$\Out_G(L)$ for all~$L$ in a counterexample to~(H-QC).
Hence we do indeed have an elimination result.

In our Theorem~\ref{theorem2}, we in fact work with the \textit{local} kernel~$H := \bigcap_{i=1}^t N_G(L_i)$,
where $L_1,\ldots, L_t$ is the~$G$-orbit of the component~$L$.
And for this local kernel~$H$,
rather than requiring~$H = F^*(G)$ as in the ``In particular'' part of Segev's Theorem~\ref{theoremSegevOriginal},
we instead ask for the inner-only condition just on the orbit of~$L$,
namely~$H = L_1 \ldots L_t C_G(L_1\ldots L_t)$---which holds for example if~$\Out_H(L) = 1$.
The behavior of components in other~$G$-orbits is hidden in~$C_G(L_1 \ldots L_t)$;
and for this centralizer, we will be able to exploit the inductive hypotheses~(H1) or~(H$L(p)$).
Moreover, the original requirement~$O_{p'}(G) = 1$ in Segev's theorems
is relaxed to just the divisibility of the order of the component~$L$ by~$p$.

Our Theorem~\ref{theorem2} will in fact follow as a special case 
of the more general and technical Theorem~\ref{theoremNormalCaseReduction}---which has still-more-flexible hypotheses.

\section{Notation and preliminaries}

In this section, we establish some notation and recall some fundamental constructions on finite groups
that will be used throughout this article.
For more details on the assertions on finite posets and their homotopy properties in relation with their order complexes,
we refer the reader to~\cite{Qui78}.
For results on finite groups we refer to~\cite{AscFGT}.
We will follow the conventions of~\cite{GL83} for finite simple groups;
the reader should be aware that there may be minor notational variations from other standard sources.

All the posets and simplicial complex considered here are finite.
If~$X$ is a finite poset, then~$\K(X)$ denotes its order complex.
Recall that the simplices of~$\K(X)$ are the non-empty chains of~$X$.
We study the homotopy properties of~$X$ by means of its order complex.
If~$f : X \to Y$ is an order-preserving map between finite posets, then~$f$ induces a simplicial map~$f : \K(X) \to \K(Y)$.
If~$f,g : X \to Y$ are two order-preserving maps between finite posets and~$f \leq g$ (i.e.~$f(x) \leq g(x)$ for all~$x \in X$),
then the induced simplicial maps~$f,g : \K(X) \to \K(Y)$ are homotopic.
Write~$X \simeq Y$ for finite posets~$X,Y$ if their order complexes~$\K(X)$ and~$\K(Y)$ are homotopy equivalent.

We denote by~$X \join Y$ the join of the posets~$X$ and~$Y$.
The underlying set of this join is the disjoint union of~$X$ and~$Y$, and the order is given as follows.
We keep the given order in~$X$ and~$Y$, and we put~$x < y$ for~$x \in X$ and~$y \in Y$.
It can be shown that~$\K(X \join Y) = \K(X) \join \K(Y)$, where the latter join is the join of simplicial complexes.
Moreover, its geometric realization coincides with the classical join of topological spaces.
That is,~if $K,L$ are simplicial complexes and~$|K|$ denotes the geometric realization of~$K$,
then we have a homeomorphism~$|K \join L| \equiv |K| \join |L|$.
For more details on these join-properties, see~\cite{Qui78}.
If~$f :X_1 \to Y_1$ and~$g : X_2 \to Y_2$ are order-preserving maps,
then we have an induced map~$f*g : X_1*X_2 \to Y_1*Y_2$
defined by~$(f*g)(x) = f(x) \in Y_1$ if~$x \in X_1$, and~$(f*g)(x) = g(x) \in Y_2$ if~$x \in X_2$.

Below we recall a generalized version of Quillen's fiber lemma (cf.~\cite[Prop.~1.6]{Qui78}. See also \cite{BWW}).
Recall that an $n$-equivalence is a continuous function $f : X \to Y$
such that $f$ induces isomorphisms in the homotopy groups $f_* : \pi_i(X) \to \pi_i(Y)$ with $i<n$,
and an epimorphism in the $n$-th homotopy group.
By the Hurewicz theorem, an $n$-equivalence also induces isomorphisms in the homology groups of degree $< n$,
and an epimorphism in degree $n$.
The topological space $X$ is $n$-connected if its homotopy groups of degree at most $n$ are trivial
(and hence its homology groups of degree at most $n$ also vanish).
By convention, $(-1)$-connected just means non-empty, and every space is~$(-2)$-connected.

\begin{proposition}[{Quillen's fiber lemma}]\label{variantQuillenFiber}
Let~$f : X \to Y$ be a map between finite posets.
Assume~$n\geq 0$.
Suppose that for all $y\in Y$, $f^{-1}(Y_{\leq y})*Y_{>y}$ (resp. $f^{-1}(Y_{\geq y})*Y_{<y}$) is $(n-1)$-connected.
Then $f$ is an $n$-equivalence.

In particular, if for all $y\in G$, $f^{-1}(Y_{\leq y})*Y_{>y}$ (resp. $f^{-1}(Y_{\geq y})*Y_{<y}$) is contractible,
then $f$ is a homotopy equivalence.
\end{proposition}

Let $X$ be a finite poset.
We denote by $\tilde{H}_*(X,R)$ the homology of $X$ with coefficients in the ring~$R$,
which is the homology of its order complex $\K(X)$.
In general we will work with $R = \QQ$ and we will just write $\tilde{H}_*(X)$.
If $f : X \to Y$ is an order-preserving map between finite posets, then we denote by~$f_*$ the map induced in homology.

If~$R$ is a field, by the K\"{u}nneth formulas,
the homology of a join of spaces is the tensor product of the homologies of the factors.
That is, we have that
\begin{equation}
\label{equationJoinHomology}
\begin{split}
\tilde{H}_*(X\join Y,R) & = \tilde{H}_*(X,R)\otimes_R \tilde{H}_*(Y,R),\\
\tilde{H}_n(X\join Y, R) & = \bigoplus_{i+j=n-1} \tilde{H}_i(X,R) \otimes_R \tilde{H}_j(Y,R).
\end{split}
\end{equation}
Note that we have a dimension shift in the join.
Roughly, it adds one degree of connectivity in the above sense.
For example, if $X,Y$ are non-empty (i.e. $(-1)$-connected), then $X*Y$ is path-connected (i.e. $0$-connected),
and if one of them is $0$-connected, then their join is simply connected (i.e. $1$-connected).
More generally, if $X$ is $n$-connected and $Y$ is $m$-connected, then~$X*Y$ is~$(n+m+2)$-connected.
For more details, see~\cite{Milnor}.
 
\bigskip

We now give some notation on finite groups, and recall some useful facts.
All the groups considered here are finite.
By a simple group we will mean a non-abelian simple group.
The alternating and symmetric group on $n$ letters are denoted by $\AA_n$ and $\SS_n$ respectively.
We also write $C_n$ and $D_n$ for the cyclic group of order $n$ and the dihedral group of order $n$, respectively.

For subgroups $H,K\leq G$, we denote by $N_K(H)$ the normalizer of $H$ in $K$, and by $C_K(H)$ the centralizer of $H$ in $K$.
We also write $[H,K]$ for the subgroup generated by the commutators between elements of $H$ and $K$.
Recall that $K$ normalizes $H$ if and only if $[H,K]\leq H$, and that~$K$ centralizes~$H$ if and only if $[H,K] = 1$.
The derived subgroup of $G$ is $G' = [G,G]$.
We denote by~$Z(G)$, $F(G)$, $O_p(G)$, $O_{p'}(G)$ the center, the Fitting subgroup,
the largest normal $p$-subgroup, and the largest normal $p'$-subgroup of $G$, respectively.

Recall that~$F(G)$ is the direct product of the subgroups~$O_p(G)$ for~$p$ a prime dividing the order of~$G$.
For solvable groups~$G$, we have the \textit{self-centralizing\/} property: that~$C_G \bigl( F(G) \bigr) \leq F(G)$.
However, this property does not hold for arbitrary groups~$G$;
so to get the desired self-centralizing property, 
we have to replace the subgroup~$F(G)$  with a natural larger subgroup~$F^*(G) = F(G)E(G)$, as follows:
The \textit{generalized Fitting subgroup} of~$G$ is the subgroup $F^*(G)$,
which is the central product of the subgroups~$F(G)$ and~$E(G)$.
The subgroup~$E(G)$ is the \textit{layer of~$G$} and it is defined as follows:
A quasisimple group is a perfect group $L$ such that $L/Z(L)$ is simple.
A \textit{component} of~$G$ is a subnormal quasisimple subgroup, and~$E(G)$ is generated by all the components of~$G$.
Note that~$G$ permutes its components via the natural conjugation action.
If~$L_1$ and~$L_2$ are distinct components of~$G$, then they commute and~$L_1 \cap L_2 \leq Z(L_1) \cap Z(L_2)$.
Therefore $E(G)$ is a central product of quasisimple groups.
We also have that $[F(G),E(G)] = 1$, and hence $F^*(G)$ is the central product of $F(G)$ and $E(G)$.
The generalized Fitting subgroup is self-centralizing: that is, we have~$C_G(F^*(G)) \leq F^*(G)$;
and when~$G$ is solvable, $F^*(G) = F(G)$.
Also~$Z(F^*(G)) = Z(F(G))$ and $Z(E(G))\leq Z(F(G))$.
See \cite[Sec.37]{AscFGT} for fuller reference.

Note that we always have~$F(G) \leq O_p(G) O_{p'}(G)$.
We will often work under the assumption that~$O_p(G) = 1 = O_{p'}(G)$, so that~$F(G) = 1$ and hence,~$Z \bigl( E(G) \bigr) = 1$.
Since~$Z \bigl( E(G) \bigr)$ equals the product of the centers of the components of~$G$,
in this case we see that the components of~$G$ are simple groups.
We summarize these standard observations (compare \cite[1.6]{AS93}) in:

\begin{lemma}
\label{lemmaOpandp}
Suppose that~$O_p(G) = 1 = O_{p'}(G)$.
Then~$F(G) = 1$ and~$F^*(G) = E(G)$ is the direct product of the components of~$G$,
which are all simple and of order divisible by~$p$.
That is, we have~$F^*(G) = L_1 \ldots L_n$; and each~$L_i$ is a simple component of~$G$.
Moreover, the self-centralizing property gives~$C_G \bigl( F^*(G) \bigr) = Z \bigl( F(G) \bigr) = 1$,
so we have a natural inclusion:
  \[ F^*(G) \leq G \leq \Aut( F^*(G) ) , \text{ and so } G/F^*(G) \leq \Out \bigl( F^*(G) \bigr) . \]
\end{lemma}

%The following lemma will be useful.

%\begin{lemma}\label{lemmaNormalizeComponent}
%If $K,L\leq G$ where $L$ is a component of $G$ and $K$ normalizes a nontrivial subgroup of $L$, then $K$ normalizes $L$.
%\end{lemma}

\noindent
Recall the outer automorphism group of~$G$ is~$\Out(G) = \Aut(G) / \Inn(G)$,
where~$\Inn(G) = G/Z(G)$ is the group of inner automorphisms of~$G$.
Recall for~$H\leq G$ that~$\Aut_G(H) = N_G(H)/C_G(H)$ is the group of automorphisms of~$H$ induced by~$G$,
and~$\Out_G(H) = N_G(H)/(H C_G(H))$ is the group of outer automorphisms of~$H$ induced by~$G$.
The subgroup~$H C_G(H)$ is the subgroup of $G$ whose elements induce inner automorphisms on~$H$.
We will say that a subgroup~$K \leq G$ induces (or acts by) outer automorphisms on~$H$,
if~$K$ normalizes~$H$, and~$K$ contains no inner automorphism of~$H$.
That is, $K$ induces outer automorphisms on~$H$ if and only if~$K\cap (HC_G(H)) = 1$.

In the following lemma, we relate the subgroups of inner automorphisms for sets of commuting subgroups of a given group.
The proof is a straightforward use of the Dedekind modular law.

\begin{lemma}\label{lemmaInnerDecomposition}
Let~$A,B \leq G$ such that~$[A,B] = 1$.
Then~$\bigl( A C_G(A) \bigr) \cap \bigl( B C_G(B) \bigr) = AB C_G(AB)$.

\noindent
Indeed if~$A_1,\ldots,A_r\leq G$ commute pairwise,
then~$\bigcap_i\ \bigl( A_i C_G(A_i) \bigr) = (A_1 \ldots A_r) C_G(A_1 \ldots A_r)$. 
\end{lemma}

In Proposition~1.6 of~\cite{AS93} it is shown that:
If~$G$ satisfies~(H1) with~$p$ odd, and does not contain components of type~$L_2(8)$, $U_3(8)$ or~$\Sz(32)$ for~$p=3,3,5$ (resp.),
then further assuming~$O_{p'}(G)\neq 1$ implies~(H-QC) for~$G$.
Inspired by this result, in \cite[Thm~1]{KP20} it is shown that the restrictions on~$p$ and on the components are unnecessary.
We recall this generalization below:

\begin{theorem}
\label{generalReduction}
Suppose that~$G$ satisfies~(H1).
If further~$O_{p'}(G) \neq 1$, then~$G$ satisfies~(H-QC).
\end{theorem}
 
 In particular this result leads to an implication relation between our inductive hypotheses:

\begin{remark}\label{remarkH1andHLp}
Suppose that~$G$ satisfies~(H1).
If~$O_{p'}(G) \neq 1$, then~$G$ has~(H-QC) by Theorem~\ref{generalReduction}.
Otherwise~$O_{p'}(G) = 1$: and then it follows %
that every component~$L$ of~$G$ has order divisible by~$p$.
Moreover, writing~$L_1 , \ldots , L_t$ for the~$G$-orbit of~$L$,
since~$L$ is nonabelian, we have~$C_G(L_1 \ldots L_t)$ is a proper subgroup of~$G$---and hence it satisfies~(H-QC) by~(H1).
That is, we obtain~(H$L(p)$). 
So:

\centerline{
  Under~(H1), $G$ satisfies either~(H-QC), or~(H$L(p)$) for every component~$L$.
            }            

\noindent
So in effect, our later proofs under~(H1) can often proceed just under~(H$L(p)$).
\donerk
\end{remark}

To finish this section, we recall the definition (for a subgroup~$H \leq G$) of the inflation~$\N_G(H)$ of~$H$,
consisting of the elements of~$\A_p(G)$ which intersect~$H$ nontrivially.
This poset had been used earlier in e.g.~\cite{Segev,SW}, and more recently in~\cite{KP20}.

\begin{definition}\label{definitionCalN}
For~$H\leq G$ and~$\B \subseteq \A_p(G)$, set:
  \[ \N_\B(H) := \{ B \in \B \tq B \cap H \neq 1 \} .\]
If~$\B = \A_p(K)$ for some subgroup~$K\leq G$, then we also write~$\N_\B(H) = \N_K(H)$.
\end{definition}

We recall some special features of this subposet:

\begin{remark}
\label{remarkPropertiesOfCalN}
We have~$\N_G(H) \simeq \A_p(H)$---using the retraction~$E \mapsto r(E) = E \cap H$,
with homotopy-inverse given by the inclusion of $\A_p(H)$ in $\N_G(H)$.
So we regard the subposet~$\N_G(H)$ as the \textit{inflation} of~$\A_p(H)$ in~$G$.

Moreover, if we consider an~$E \in \A_p(G) - \N_G(H)$ (that is, with~$E \cap H = 1$), 
then~$r$ restricts to a homotopy equivalence~$\N_G(H)_{>E} \simeq \A_p \bigl( C_H(E) \bigr)$,
with homotopy-inverse given by~$A \mapsto AE$.
So when we consider the fibers of the inclusion~$i : \N_G(H) \subseteq \A_p(G)$,
for~$E \in \A_p(G) - \N_G(H)$ we get:
  \[ i^{-1}( \A_p(G)_{\geq E} ) = \N_G(H)_{>E} \simeq \A_p \bigl( C_H(E) \bigr) . \]
This relation connects Quillen's fiber lemma (Proposition~\ref{variantQuillenFiber}) with the study of centralizers.
\donerk
\end{remark}

\section{Discussion: Segev's method, via a Mayer-Vietoris argument}
\label{sectionDiscussionSegev}

The proof of our main technical result Theorem~\ref{theoremNormalCaseReduction}
(which leads to our more concrete elimination-results
such as Theorems~\ref{theorem2} and Corollary~\ref{cor:theorem2MoreGeneral})
uses a generalization of Segev's method in~\cite{Segev}.
In this section, we outline some fundamental aspects of Segev's proofs;
and along the way, we indicate where we can make our extensions.

\bigskip

\noindent
We begin the section with some general background for Segev's theorems:

We are pursuing results which establish~(H-QC).
So:

\centerline{
We assume~$O_p(G) = 1$;
            }

\noindent
and we need to obtain:

\centerline{
  (Goal) \quad $\tilde{H}_* \bigl( \A_p(G) \bigr) \neq 0$.
            }

\noindent
Segev's arguments primarily involve:

\smallskip

\centerline{
  $H :=$ the kernel of the $G$-conjugation action on its components.
            }

\begin{remark}[Some consequences of~$O_{p'}(G) = 1$]
\label{rk:conseqOp'G=1}
Segev works under the further assumption %
that~$O_{p'}(G) = 1$;
and using this along with~$O_p(G) = 1$, he gives some standard consequences at~pp955--956 of~\cite{Segev}.
We summarize a number of such consequences (cf.~our~\ref{lemmaOpandp}) as:

  (i) $F^*(G) = E(G)$ is a direct product of simple components~$L_1, \dots, L_n$, with~$G$ faithful on~$F^*(G)$;

  (ii) The components of~$G$ have order divisible by~$p$.

\noindent
Segev's arguments seem to frequently depend on these properties---but very often only {\em implicitly\/}.
For our analysis of his logic with an eye to generalization,  we will always try to be much more explicit.
For example, we mention that using~(ii) above we get: 

  (iii) $\A_p(L_i) \neq \emptyset$ for all $i$, and hence $\A_p(H)\neq\emptyset$.

\noindent
This fact seems to be fundamental---though unstated---throughout~\cite{Segev}. 
\donerk
\end{remark}

A further point which seems at best implicit in~\cite{Segev}
is the handling of the case~$n = 1$---namely where there is just a single component~$L_1$:
Note using Remark~\ref{rk:conseqOp'G=1}(i) above that then~$F^*(G) = L_1$, so that~$G$ is almost-simple. 
Now the well-known result of Aschbacher and Kleidman in~\cite{AK90} gives~(H-QC) for~$G$ almost-simple;
so that result {\em could\/} be quoted to finish when~$n = 1$. 
But in fact, Segev only explicitly quotes~\cite{AK90} on~p956: and there,
only to show (for any~$n$) that the hypothesis of Theorem~3 in~\cite{Segev} satisfies the hypothesis of Theorem~2 there.
Thus it seems likely that a different argument was intended to cover the case~$n = 1$:
For notice that this condition in particular gives
a special case of trivial~$G$-conjugation action on components---that is, $H = G$.
So we will include the reduction to~$n \geq 2$ as Remark~\ref{rk:conseqH<G}(ii) below,  
during our discussion of the more general trivial-action situation:
where we will see that reduction to~$H < G$ provides the foundation for the ``generic-context'' in our logic-analysis, namely:

\bigskip

\noindent
{\em A Mayer-Vietoris sequence, based on nontrivial action on components:}%

\bigskip

\noindent
Before starting on the proofs, we develop an outline for the overall approach followed in them.

The more specific hypotheses of Theorems~1 and~2 in~\cite{Segev} involve restrictions related to the homology of~$\A_p(H)$.
And the proofs there seem to begin by assuming that the case~$H < G$ must hold.%
\footnote{
   In fact, an unpublished result of Thompson (in a preprint on ``separating sets'', from July~1991)
   shows that a counterexample~$G$ to~(H-QC) under the hypotheses of Aschbacher-Smith ~\cite{AS93}
   should have~$H < G$---indeed~$G/H$ should be nonsolvable.
   But we won't in this paper follow that direction:
   since we are instead focusing on the kernel~$H$ of the action, and outer automorphisms that it induces on components.
          }
So for completeness, when we later outline the logic in those proofs, 
we will at those points add a treatment of the seemingly-omitted case where~$H = G$---that is,
where~$G$ normalizes all the components. %
Indeed, we will treat the slightly more general situation where~$\A_p(H) = \A_p(G)$:
for since~(Goal) only involves~$\A_p(G)$ (via its homology),
we may as well assume in that situation that~$H = G$.
In brief summary:
We will verify there that in the trivial-action case,
the hypotheses of Theorems~1 and~2 essentially already {\em contain\/} the conclusion~(H-QC); 
so that no further proof is then actually required. 

Hence, in the remainder of our introductory Mayer-Vietoris discussion, i.e.~through Remark~\ref{rk:initimpl}:

\centerline{
  We assume temporarily that we have already reduced to~$\A_p(H) \subsetneq \A_p(G)$;
            }

\noindent
and hence also to~$H < G$---namely nontrivial action on components. %
In this situation,
we want to set up a general context for the main arguments, within this ``generic'' case of nontrivial action.

In order to work toward the nonzero homology for~$\A_p(G)$ in~(Goal),
we will proceed by decomposing that homology over certain other~$p$-subgroup posets related to~$H$.
In particular, we will make use of the inflation~$Y := \N_G(H)$ of~$\A_p(H)$,
and the complement~$Z := \A_p(G) - \A_p(H)$ to~$\A_p(H)$;
that is, we will work with the decomposition~$\A_p(G) = Y \cup Z$.
We set~$Y_0 := Y \cap Z$ for the overlap.
We also recall by Remark~\ref{remarkPropertiesOfCalN}
that~$Y$ is homotopy equivalent to~$\A_p(H)$, via the retraction~$r(E) = E \cap H$;
we then further set~$V_0 := r(Y_0)$, and let~$b : V_0 \subseteq \A_p(H)$ denote the inclusion.
We will need:

\begin{remark}[Some consequences of reducing to nontrivial action~$\A_p(H) \subsetneq \A_p(G)$]
\label{rk:conseqH<G}
Recall we are assuming here that we have in fact reduced to~$\A_p(H) \subsetneq \A_p(G)$. 
Let's examine the effect of this assumption on the terms of the above decomposition $\A_p(G) = Y \cup Z$.

Notice first that~$Y = \N_G(H)$~is nonempty---for it contains~$\A_p(H)$,
which we saw at Remark~\ref{rk:conseqOp'G=1}(iii) is nonempty.

Next note that our assumed reduction provides us with some~$A \in \A_p(G) - \A_p(H) = Z$, 
so that~$Z$ also is nonempty. 
In fact, we get~$A = (A \cap H) \times B$, 
where nontrivial~$B \in \A_p(G)$ has~$B \cap H = 1$---so that~$B \in Z - Y$.
And then our decomposition~$Y \cup Z$ is {\em nontrivial\/}---that is,
we really do get {\em two\/} independent nonempty parts.
Indeed we can add here, again using~$\A_p(H) \neq \emptyset$, 
that any~$E \in Z$ must centralize some~$F \in \A_p(H)$---giving a member~$EF$ of~$Y_0$, since~$EF \cap H \geq F > 1$;
so that also~$Y_0$ is nonempty, and hence~$V_0$ is nonempty.
We summarize these observations in:

  (i) The decomposition~$\A_p(G) = Y \cup Z$ is nontrivial; with~$Y,Z,Y_0,V_0$ nonempty.

\noindent
Note also that from the assumed reduction to~$H < G$, we immediately get:

  (ii) We have~$n \geq p \geq 2$ for the number of components of~$G$.

\noindent
These observations (and indeed some other consequences of~$H < G$)
will be important at various  points, as our analysis below continues.
\donerk
\end{remark}

In view of our assumed reduction, we now see via Remark~\ref{rk:conseqH<G}(i) 
that our fundamental decomposition~$\A_p(G) = Y \cup Z$ is nontrivial. 
And in using the decomposition, 
we can roughly regard the inflation~$Y$ as the neighborhood of~$\A_p(H)$,
and~$Y_0 = Y \cap Z$ as the boundary of that neighborhood.
In this viewpoint, the condition~(A) below exploits a local-requirement related to~$H$:
namely that the boundary~$Y_0$ should (perhaps not surprisingly) be ``thinner'' than the full neighborhood~$Y$, 
in the sense of not being surjective in homology;
and then from this restriction,
we will obtain the (perhaps surprising) global-consequence of nonzero homology for~$\A_p(G)$.
To implement this plan, we will use the natural homological context for our decomposition:

\begin{remark}[The Mayer-Vietoris sequence for the decomposition~$\A_p(G) = Y \cup Z$]
\label{rk:MVseqfordecompYZ}
The Mayer-Vietoris exact sequence for our decomposition~$\A_p(G) = Y \cup Z$ takes the form:
  \[ \ldots \to \tilde{H}_{k+1} \bigl( \A_p(G) \bigr)
            \to \tilde{H}_{k}(Y_0)
	    \overset{\alpha}{\to} \tilde{H}_k(Y) \oplus \tilde{H}_k(Z)
	    \overset{\beta}{\to}  \tilde{H}_k \bigl( \A_p(G) \bigr)
	    \to \ldots
  \]
Recall via our assumed reduction to~$\A_p(H) \subsetneq \A_p(G)$ and Remark~\ref{rk:conseqH<G}(i) that~$Y,Z \neq \emptyset$.
\donerk
\end{remark}

\noindent
Note that if we had~$\A_p(H) = \A_p(G)$ in Remark~\ref{rk:MVseqfordecompYZ} above, we would get~$Z = \emptyset$;
this then leads to~$Y = \A_p(H) = \A_p(G)$, and hence we get the trivial decomposition~$\A_p(G) = Y \cup Z = \A_p(G) \cup \emptyset$---in
which case, the sequence would degenerate to just an isomorphism of~$\tilde{H}_* \bigl( \A_p(G) \bigr)$ with itself.
(This why our development has emphasized the need to obtain nontriviality for the decomposition;
the issue of nontriviality is seemingly omitted, when the sequence is introduced at~p958 of~\cite{Segev}.)

Recall~$Y_0 \neq \emptyset$ by Remark~\ref{rk:conseqH<G}(i).
In the context of the Mayer-Vietoris sequence in Remark~\ref{rk:MVseqfordecompYZ},
our approach to~(Goal) will be to now establish, for the inclusion~$a : Y_0 \subseteq Y$, that:
\begin{quote}
  (A) \quad The induced map in homology~$a_* : \tilde{H}_*(Y_0) \to \tilde{H}_*(Y)$ is not surjective.
\end{quote}

\noindent
For since~$a_*$ gives the~$Y$-coordinate of the value of~$\alpha$ in the sequence,
we see under~(A) that then~$\alpha$ also is non-surjective: that is, has proper image.
But by exactness, $\text{im}(\alpha) = \text{ker}(\beta)$ is then also proper, 
and this proper kernel gives~$\beta$ the desired nonzero image in~$\tilde{H}_* \bigl( \A_p(G) \bigr)$.

To now prove~(A), we usually proceed somewhat indirectly---via examining further related maps in homology.
As we continue our analysis below, we will build up successive homological terms, and maps among them;
these are collected in our ``omnibus'' commutative diagram, in later Remark~\ref{rk:diagramSegev}---as
the culmination of our analysis of the overall logic for proving Theorems~1 and~2.

\bigskip

\noindent
The first few such ``related maps'' are reasonably self-evident:

Recall by Remark~\ref{remarkPropertiesOfCalN} that~$Y$ is homotopy equivalent to~$\A_p(H)$, via the retraction~$r(E) = E\cap H$;
so condition~(A) is equivalent to:
\begin{quote}
  (A$^\prime$)
     \quad The composition
           map~$\tilde{H}_*(Y_0) \overset{a_*}{\to} \tilde{H}_*(Y) \overset{r_*}{\to} \tilde{H}_* \bigl( \A_p(H) \bigr)$
           is not surjective. 
\end{quote}
Now recall by Remark~\ref{rk:conseqH<G}(i) that~$V_0 = r(Y_0) \neq \emptyset$. 
We can formulate a condition analogous to~(A):
\begin{quote}
  (B) \quad The induced map in homology~$b_* : \tilde{H}_*(V_0) \to \tilde{H}_*(\A_p(H))$ is not surjective.
\end{quote}
And we see under~(B) that non-surjectivity of~$b_*$
gives non-surjectivity also for the composition given by~$b_* \circ (r|_{Y_0})_*$---which
(cf.~the diagram in Remark~\ref{rk:diagramSegev}) is the composition~$r_* \circ a_*$, so that we in fact get~(A$^\prime$).
We summarize these deductions in:

\begin{remark}[Some initial implication-relations]
\label{rk:initimpl}
So far we have obtained the logical implications: 
\begin{center}
  (B) $\Rightarrow$ (A$^\prime$) $\Leftrightarrow$ (A) $\Rightarrow$ (Goal) .
\end{center}
In particular, we see that one route toward (Goal) is via establishing~(B).
\donerk
\end{remark}

To now prove~(B), we will need suitable further hypotheses on the homology of~$\A_p(H)$.
Some of these will lead us to formulate more conditions and implications---in particular extending,
by an additional term implying~(B), the chain in Remark~\ref{rk:initimpl} to that in later Remark~\ref{rk:summarylogimplic}.
But first, in the proof of Theorem~1 in~\cite{Segev} below, we will instead fairly directly obtain~(B) itself:

\bigskip
\vspace{0.2cm}

\noindent
{\em An outline of the proof of Segev's Theorem~1:} 

\bigskip

For Theorem~1, the ``suitable further hypothesis'' is that the inclusion:

\centerline{
   $i : \D_p(H) \subseteq \A_p(H)$
            }

\noindent
should induce a map~$i_*$ in homology which is non-surjective:
where~$\D_p(H)$ is the ``diagonal subposet'' of~$H$ defined on~\cite[p956]{Segev}.
We use the equivalent definition that~$\D_p(H)$ is given by the~$A \in \A_p(H)$
such that there exist components~$L_1, \ldots , L_t$, $t \geq 2$, with~$C_A(L_1 \cdots L_t) = C_A(L_i)$ for all~$i$.

The proof of Theorem~1 in~\cite{Segev} begins with the decomposition~$\A_p(G) = Y \cup Z$,
as in the setup for Remark~\ref{rk:conseqH<G} above;
thus in effect, it assumes that the case~$Z \neq \emptyset$ must hold.
So, in order to cover the seemingly-omitted case~$Z = \emptyset$, 
we give here our promised explicit treatment of the (equivalent) case~$\A_p(H) = \A_p(G)$ of trivial action.
Note that in any case for~$\A_p(H)$, the hypothesis in Theorem~1 of non-surjectivity of the map~$i_*$%
\footnote{
   We mention that this reduction to~$H < G$ using~$i_*$ is valid,
   even when the domain-poset~$\D_p(H)$ for~$i$ is empty---so that reduced homology in formal dimension~$-1$ is being used. 
   But the reader uncomfortable with arguments using the empty set could instead cover that empty-subcase
   by using our later observation (during our analysis of the proof of Theorem~2) that~$\D_p(H) = \emptyset$ gives~$n = 1$;
   and then in that situation, where~$G$ is almost-simple, finishing via~(H-QC) using~\cite{AK90} as we had mentioned earlier.
          }
implies that its codomain-space~$\tilde{H}_* \bigl( \A_p(H) \bigr)$ must be nonzero.
So when~$\A_p(H) = \A_p(G)$, we in fact have the nonzero homology of~$\A_p(G)$ required for~(Goal).
Thus for this case, the hypothesis already {\em includes\/} the conclusion~(H-QC) of the result.
That is, Theorem~1 becomes essentially a tautology here: so that this subcase doesn't really {\em require\/} further proof.
However, the overall proof of Theorem~1 seems incomplete, without some mention of this ``not-required'' fact.

Thus we have explicitly reduced to~$\A_p(H) \subsetneq \A_p(G)$ (and hence nontrivial action~$H < G$).
Consequently by Remark~\ref{rk:conseqH<G}(i):
we do indeed have nontriviality of the decomposition~$\A_p(G)$, including~$Z \neq \emptyset$; 
and thus we may use the generic-context of the Mayer-Vietoris sequence given in earlier Remark~\ref{rk:MVseqfordecompYZ}. 

Within that Mayer-Vietoris context,
and the corresponding terms appearing in the left half of the summary-diagram in Remark~\ref{rk:diagramSegev},
the crucial step in the proof is that:
\begin{quote}
  We get an inclusion~$j : V_0 \subseteq \D_p(H)$.
\end{quote}
We see then that the non-surjectivity hypothesis on~$i_*$
now gives non-surjectivity also for the composition~$i_* \circ j_*$---which is the map~$b_*$:
that is, we get~(B)---and hence~(Goal) using the implications in Remark~\ref{rk:initimpl}.

\bigskip

We mention that the proof of this inclusion~$j$ on~p959 of~\cite{Segev} begins
by assuming that the case~$V_0 \neq \emptyset$ holds. 
The seemingly-omitted case~$V_0 = \emptyset$ can in fact be covered by invoking our reduction to~$\A_p(H) \subsetneq \A_p(G)$,
as we saw in Remark~\ref{rk:conseqH<G}(ii).

For later comparison, we also indicate here the rest of the proof of that inclusion~$j$:
As just indicated, we can indeed assume that~$V_0 \neq \emptyset$; so consider any~$B \in V_0$.
Then we have~$B = A \cap H$ for some~$A \in Y_0$, where~$A$ contains some $a \not \in H$. 
Since~$H$ is the kernel on components,
$a$ must have at least one nontrivial orbit~$\O$ (of size~$p \geq 2$) on components;
and then the nontrivial elements of~$B \leq C_G(a)$ must exhibit isomorphic actions on all members~$L$ of~$\O$---which
gives us the condition (namely~$C_B(\O) = C_B(L)$ for all~$L \in \O$) defining~$B \in \D_p(H)$.

\bigskip

We turn now to arguments which {\em will\/} involve further conditions in order to lead to~(B):

\bigskip

\noindent
{\em An outline of the proof of Segev's Theorem~2:} 

\bigskip

These further conditions arise in the proof of Theorem~2 in~\cite{Segev},
which we stated earlier as Theorem~\ref{theoremSegevOriginal};
in fact we will indicate a slight extension of Segev's argument, which we will use in our results later in this paper.

We recall that for Theorem~2, the ``suitable further hypothesis on~$H$''
is that for each component~$L_i$ of~$G$, 
the inclusion~$\A_p(L_i) \subseteq \A_p \bigl( \Aut_H(L_i) \bigr)$ should induce a {\em nonzero\/} map in homology. 
(The faithful action of~$L_i$ here implicitly uses simplicity of~$L_i$ in Remark~\ref{rk:conseqOp'G=1}(i).)

Segev's proof of Theorem~2 in fact involves the join~$X$ of factors
which are given by the above posets~$\A_p \bigl( \Aut_H(L_i) \bigr)$;
along with some other associated constructions.  
In our generalization of his method in Section~\ref{sectionMainTool}, 
we will want to instead choose~$X$ as the join of certain variants on those factors. 
So in order to focus on some of the issues involved, 
before we begin the main proof of Theorem~2, 
we will first spend a few paragraphs sketching various aspects involved in working with such joins.
We will give an overall context a little more general than in~\cite{Segev};
though we will still indicate the features given by Segev at several different points in that paper.

\bigskip

\noindent
{\em Preliminary discussion: Constructions related to the join-poset~$X$\/}.
In order to exploit the posets in the hypothesis of Theorem~2, Segev on~p957 of~\cite{Segev} defines their join:%

\begin{definition}[The fundamental join~$X$]
\label{defn:Kasjoin}
We set: 

\centerline{
  $X := X_1 * \dots * X_n$; where~$X_i := \A_p \bigl( \Aut_H(L_i) \bigr)$.
            }

\noindent
This~$X$ provides the codomain for the crucial poset map $\psi : \A_p(H) \to X$ indicated below.
\donerk
\end{definition}

\noindent
Various related structures are general; i.e.~they do not {\em yet\/} involve our hypothesis on the~$X_i$:

There are several direct-product groups in the background (cf.~p956 of~\cite{Segev}).
Using Remark~\ref{rk:conseqOp'G=1}(i), we have~$E(G) = L_1 \times \cdots \times L_n \leq H$,
with~$H$ acting faithfully on~$E(G)$.
We let~$\pi_i : H \to \Aut_H(L_i)$ be the natural projection;
and let~$\pi := (\pi_1 ,\ \dots\ , \pi_n)$ denote
the ``co-ordinate'' map into the formal direct product~$J = J_1 \times \ldots \times J_n$, 
where $J_i := \Aut_H(L_i)$.
From faithfulness, it is standard that~$\pi$ maps~$H$ to an isomorphic subgroup~$\pi(H) \leq J$.
That is: though~$H$ itself need not admit a direct-product decomposition, 
it is in effect embedded in the abstractly-given formal direct product~$J$ of the individual groups~$\Aut_H(L_i) = J_i$.

It is also standard that a direct product of groups leads to certain analogous join-relationships at the level of~$\A_p$-posets:
First~$\A_p \bigl( E(G) \bigr)$ is homotopy equivalent to the join~$\B$ of the individual~$\A_p(L_i)$.
Similarly~$\A_p(J)$ is homotopy equivalent to the join~$X$ of the individual~$\A_p(J_i) = X_i$;
and we write~$\psi$ for this standard homotopy equivalence (compare~\cite[Prop~2.6]{Qui78}).
We will describe~$\psi$ more precisely at later Definition~\ref{defn:iEandpsi}.
The restriction of this standard~$\psi$ to~$\A_p \bigl( \pi(H) \bigr)$
is essentially the map which Segev denotes by~``$\psi : \A_p(H) \to X$'' ---and
we will similarly write~$\psi$ rather than $\psi |_{\A_p(H)} = \psi \circ \pi$ for this restriction.
Indeed within the join~$X$, we get a relationship analogous with that in the previous paragraph:
namely~$\psi$ maps~$\A_p(H)$ to~$\psi \bigl( \A_p(H) \bigr) \subseteq X = X_1 * \dots * X_n$.
That is, though~$\A_p(H)$ itself need not admit a join-decomposition,
it has a natural image 
in the abstractly-given formal join~$X$ of the~$\A_p \bigl( \Aut_H(L_i) \bigr) = X_i$. 

Finally we recall the standard fact (cf.~our~(\ref{equationJoinHomology}), or~\cite[(2.2)]{Segev})
that the reduced homology of a join such as~$X$ is the tensor product of the reduced homology of the factors~$X_i$.

\bigskip

\noindent
With these features in place, we now {\em do\/} apply our specific hypothesis on components:

Namely the homology of the join~$\A_p \bigl( E(G) \bigr)$ is the tensor product of the homology of the factors~$\A_p(L_i)$;
and our hypothesis says that the ``inclusions''~$\A_p(L_i) \subseteq \A_p \bigl( \Aut_H(L_i) \bigr) = X_i$ are nonzero in homology.
Hence, the join gives an inclusion map~$\A_p(L_1) * \ldots * \A_p(L_n) = \B \hookrightarrow X$
which is also nonzero in homology.
Then the composition:
  $$ \psi|_{\A_p( E(G) )}: \A_p \bigl( E(G) \bigr) {\to} \A_p(L_1) * \ldots * \A_p(L_n) {\hookrightarrow} X $$
is not the zero map in homology, since~$\psi|_{\A_p( E(G) )}$ is a homotopy equivalence with its image.
Since~$\psi |_{\A_p( E(G) )} = \psi \circ e$ for $e : \A \bigl( E(G) \bigr) \subseteq \A_p(H)$,
we see that~$\psi$ also is nonzero in homology.
Thus we have established:

\begin{remark}[$\A_p(H)$-form of hypothesis for Theorem~2]
\label{rk:ApHformhypThm2}
For the join~$X$ of the~$X_i = \A_p(\Aut_H(L_i))$:

\hfil  Under the hypothesis of Theorem~2, $\psi : \A_p(H) \to X$ induces a nonzero map in homology. \hfill \donerk
\end{remark}

\noindent
This~$\A_p(H)$-form of the hypothesis appears as~(4.2)(3) in~\cite{Segev};
and also as~(D) in our main analysis for Theorem~2 below.
And it is the form that is actually used in the proof of Theorem~2.

\bigskip

\noindent
We complete our preliminary discussion with several further standard features related to~$X$.

We will need for certain later arguments the precise definition of the above poset isomorphism~$\psi$:

\begin{definition}[The map~$\psi : \A_p(H) \to X$]
\label{defn:iEandpsi}
First define for $E\in\A_p(H)$:

\centerline{
   $i_E :=$ the largest index~$i$, such that~$\pi_i(E) \neq 1$ (i.e.~$E \nleq C_H(L_i)$).
            }

\noindent
Then (cf.~\cite[p960]{Segev}) we select {\em only\/} the corresponding projection~$\pi_{i_E}$ for application to~$E$:
that is, we define~$\psi(E) := \pi_{i_E}(E)$; this image of course lies in~$\A_p(\Aut_H(L_{i_E})) = X_{i_E} \subseteq X$.
\donerk
\end{definition}

\noindent
We emphasize in particular that the term~$X_{i_E}$ of the join~$X$ which contains~$\psi(E)$ is uniquely determined:
since for any other~$i \neq i_E$, we do {\em not\/} use the projection~$\pi_i(E) \in X_i$.
Also:
It is standard that the particular homotopy equivalence~$\psi$ depends on the ordering of the factors of the direct product;
but of course any chosen ordering does give such a homotopy equivalence.

Next (cf. \cite[p957]{Segev})
the order complex~$K:=\K(X)$ of the join~$X$ determines some useful standard subcomplexes, as follows:

\begin{definition}[$K_0$ and~$\hat{K}_0$]
\label{defn:K0hatK0}
For a fixed~$i$, let~$X_{\hat{i}}$ denote the maximal sub-join determined by the~$X_j$ for~$j \neq i$, 
and then take: 
\begin{center}
  $K_0 := \bigcup_{i=1}^n\ \K( X_{\hat{i}})$ .
\end{center}
Notice that~$K_0 \neq \emptyset$ if and only if~$n \geq 2$ (see also Remark \ref{rk:conseqOp'G=1}(iii)).
The chains~$c$ in~$K_0$ are ``sub-maximal'', in the sense that~$c$ determines at least one index~$i_c$,
such that the members of~$c$ do not lie in~$X_{i_c}$.

Finally (for any~$n \geq 1$) we indicate a standard contractible (cf.~\cite[p958]{Segev}) subcomplex~$\hat{K}_0$ of~$K=\K(X)$:
We first fix some arbitrary vertex~$v_i \in X_i$ for each~$i$.
(We have~$X_i \neq \emptyset$ by Remark~\ref{rk:conseqOp'G=1}(iii).)
Then we set:
\begin{center}
  $\hat{K}_0 := \bigcup_{i=1}^n\  \St_K(v_i)$ ,
\end{center}
where~$\St$ denotes the topological star.
Notice in particular that when~$n \geq 2$ (so that~$K_0 \neq \emptyset$ as noted above), 
we get~$\K(X_{\hat{i}}) \subseteq \St_{\K(X_i)}(v_i) * \K(X_{\hat{i}}) = \St_K(v_i) \subseteq \hat{K}_0$,
so that we have an inclusion~$f : K_0 \subseteq \hat{K}_0$. 
Note then that for~$c : K_0 \subseteq K$, we have~$c = d \circ f$, where $d : \hat{K}_0 \subseteq K$. 
Here the contractibility of~$\hat{K}_0$ guarantees that~$d$ is zero in homology;
so that also~$c$ is zero in homology.
\donerk
\end{definition}

\noindent
We note that~$K=\K(X)$, $K_0$, and~$c$ are all needed in our summary diagram in later Remark~\ref{rk:diagramSegev}.

\bigskip

With the above preliminary discussion of constructions related to~$X$ completed,
we begin the main proof of Theorem~2:

\bigskip

\noindent
{\em The main logical analysis for the proof of Theorem~2\/}.
The proof of Theorem~2 in~\cite{Segev} begins on p960 by setting up for a reduction to the hypothesis of Theorem~1
(and hence in effect to condition~(B) of Remark~\ref{rk:initimpl}).
We will instead indicate a slight generalization of the argument, which leads more directly to (B).
By either route, the argument depends in several ways (as was the case for Theorem~1)
on having made the reduction to the nontrivial-action case~$\A_p(H) \subsetneq \A_p(G)$.  

So again we give the earlier-promised explicit discussion 
of the seemingly-omitted trivial-action case~$\A_p(H) = \A_p(G)$. 
Here, we can simply use the~$\A_p(H)$-form of the hypothesis for Theorem~2 that we derived in Remark~\ref{rk:ApHformhypThm2}:
Note that in any case for~$\A_p(H)$, 
nonzero-ness of the induced map~$\psi_*$ there forces nonzero-ness of its domain-space~$\tilde{H}_* \bigl( \A_p(H) \bigr)$.
So when~$\A_p(H) = \A_p(G)$, we in fact have the nonzero homology of~$\A_p(G)$ required for~(Goal).
Thus for this case, the hypothesis already {\em includes\/} the conclusion~(H-QC) of the result.
That is, Theorem~2 becomes essentially a tautology here: so that this case doesn't really {\em require\/} proof.
However, again the overall proof of Theorem~2 seems incomplete, without some mention of this ``not-required'' fact.

We also now recall from Remark~\ref{rk:conseqH<G}(ii)
that our reduction here to~$\A_p(H) \subsetneq \A_p(G)$ also reduces us to the case of~$n \geq 2$ components.
This latter reduction covers for example the seemingly-omitted case of~$n = 1$ at~p957 in~\cite{Segev}, 
where~$n \geq 2$ is assumed to hold throughout that paper; 
and similarly we saw in Definition~\ref{defn:K0hatK0}
that the reduction to~$n \geq 2$ covers the case of~$K_0 = \emptyset$, which is seemingly omitted on~p958.

Since we have reduced to the nontrivial-action case~$\A_p(H) \subsetneq \A_p(G)$, then just as for Theorem~1,
by Remark~\ref{rk:conseqH<G}(i) we again get nontriviality of the decomposition $\A_p(G) = Y \cup Z$;
and we can indeed use the generic-context of the Mayer-Vietoris sequence as in Remark~\ref{rk:MVseqfordecompYZ}.

This time, we will investigate the homology of the right-hand terms
in the summary-diagram in Remark~\ref{rk:diagramSegev}---namely~$X$ and~$K_0$,
constructed in Definitions~\ref{defn:Kasjoin} and~\ref{defn:K0hatK0}.
We reproduce, as condition~(D) below, the $\A_p(H)$-form of the hypothesis for Theorem~2,
which we had obtained in Remark~\ref{rk:ApHformhypThm2} above.
Our discussion of~$\hat{K}_0$ in Definition~\ref{defn:K0hatK0} corresponds roughly to Segev's at~p958:
in particular we observed there that the inclusion~$c : K_0 \subseteq \K(X)$ is zero in homology,
and we reproduce this as~(E) below.
At this point, Segev can finish by establishing the further containment~(C):
\begin{quote}
(C) \quad $\psi \bigl( \K(V_0) \bigr) \subseteq K_0$.
\end{quote}
\begin{quote}
(D) \quad The induced map~$\psi_* : \tilde{H}_* \bigl( \A_p(H) \bigr) \to \tilde{H}_*(X)$ is nonzero.
\end{quote}
\begin{quote}
(E) \quad The induced map~$c_* : \tilde{H}_*(K_0) {\to} \tilde{H}_*(X)$ is zero.
\end{quote}
For note then that we get~$c_* \circ\ ( \psi |_{V_0} )_* = 0$ using~(C + E);
so by commutativity in the diagram in Remark~\ref{rk:diagramSegev} below, we get~$\psi_* \circ\ b_* = 0$---and then
nonzero-ness of~$\psi_*$ in~(D) forces non-surjectivity of~$b_*$, as required for~(B).
Thus by the implications in Remark~\ref{rk:initimpl}, we get~(Goal), completing the proof of Theorem~2.

\begin{remark}[Summary of logical implications]
\label{rk:summarylogimplic}
The previous paragraph in particular exhibits a further implication-relation among our various conditions---so that
we can now write:
\begin{center}
(C + D + E) $\Rightarrow$ (B) $\Rightarrow$ (A$^\prime$) $\Leftrightarrow$ (A) $\Rightarrow$ (Goal) ;
\end{center}
extending our earlier sequence in Remark~\ref{rk:initimpl}.
\donerk
\end{remark}

The maps involved in these relations are indicated in:

\begin{remark}[Overall commutative diagram]
\label{rk:diagramSegev}
We have:

\begin{equation*}
\xymatrix{
\tilde{H}_*(Y_0) \ar[d]_{a_*} \ar[r]^{{(r|_{Y_0})}_*} & \tilde{H}_*(V_0) \ar[d]_{b_*} \ar[r]^{ {(\psi|_{V_0})}_*}
                                                      & \tilde{H}_*(K_0) \ar[d]^{=0}_{c_*} \\
\tilde{H}_*(Y) \ar[r]^{r_*}_{\groupiso}               & \tilde{H}_*(\A_p(H)) \ar[r]^{\psi_*}_{ \neq 0}
                                                      & \tilde{H}_*(K)
          }
\end{equation*}
where in the right half, we have included the desired properties of the maps given in~(D,E).
\donerk
\end{remark}

\noindent
To our outline above of the proof of Theorem~2, we now add some comments on certain details: 

Segev's proof does not explicitly use the inclusion~$\psi \bigl( \K(V_0) \bigr) \subseteq K_0$ that we indicated in~(C); 
instead he establishes in~(4.3) of~\cite{Segev} %
just the terminal-segment~$\psi(\ \K \bigl( \D_p(H) \bigr)\ ) \subseteq K_0$ of that inclusion.
His proof there begins by assuming that the case~$\D_p(H) \neq \emptyset$ must hold; 
so we indicate an explicit treatment to eliminate the seemingly-omitted case where~$\D_p(H) = \emptyset$:
Namely via our earlier reduction to nontrivial-action~$H < G$,
we may (as noted in Remark~\ref{rk:conseqH<G}(ii)) assume that we have~$n \geq 2$ components---so that
we can consider a pair of distinct components---say~$L_1,L_2$. 
We had also seen using Remark~\ref{rk:conseqOp'G=1}(ii) that components have order divisible by~$p$;
so taking~$a_i$ of order~$p$ in each~$L_i$, we have $A := \langle a_1 a_2 \rangle \in \A_p(H)$.
But then~$1 = C_A(L_1 L_2) = C_A(L_i)$ for each~$i$, giving the condition defining~$A \in \D_p(H)$.
That is, we may indeed assume that~$\D_p(H) \neq \emptyset$. 

Segev then finishes the proof as follows: 
In the final commutative-diagram argument above using~(C + E), 
he in effect begins at~$\D_p(H)$ rather than~$V_0$---and obtains instead non-surjectivity for~$i_*$
induced by the inclusion~$i : \D_p(H) \subseteq \A_p(H)$---that is, the hypothesis for Theorem~1. 
Thus he obtains~(H-QC) in Theorem~2 as corollary of his Theorem~1.  
(So the use of condition~(B) for the proof of Theorem~2 remains implicit--i.e.~via the proof of Theorem~1.)

By contrast, our extension above of his argument obtains the inclusion~$\psi \big( \K(V_0) \bigr) \subseteq K_0$ in~(C)
by applying his inclusion~$\psi(\K \bigl( \D_p(H) \bigr)) \subseteq K_0$ in Theorem~2
after his earlier inclusion~$j : V_0 \subseteq \D_p(H)$ from the proof Theorem~1. 
Thus we are quoting only part of the {\em argument\/} for Theorem~1, as opposed to Segev's quoting the {\em result\/}.
To see that our use of~$V_0$ and~$j$ from that argument in Theorem~1 is in fact valid for Theorem~2,
we need to recall that that earlier argument depended on having~$V_0 \neq \emptyset$, via the reduction to~$H < G$---and
we did in fact we also obtain~$H < G$ independently for the proof of Theorem~2.
And note furthermore that the inclusion-argument for~$j$ in Theorem~1 depended only on the definition of~$\D_p(H)$---and
not on the specific restriction on~$\D_p(H)$ in the hypothesis of Theorem~1.
So we can indeed use that inclusion also in the proof of Theorem~2. 

\bigskip

We conclude the section with:

\bigskip

\noindent
{\em Further remarks on generalizing Segev's methods\/}

\bigskip

In our generalizations starting in the next section, 
we will proceed via essentially the same sequence of implications as indicated in Remark~\ref{rk:summarylogimplic} above.
However, in contrast with Segev, we will now mainly take~$H$ to be the ``local'' version of the kernel on components---that is,
the kernel of the permutation action on the~$G$-orbit of a single component.
So we will make appropriate adjustments to the constructions involved;
notably to the factors in the join~$X$.

The diagonal subposet~$\D_p(H)$ will not usually be involved;
instead, we will want to establish the inclusion in~(C) ``directly'';
that is, without proceeding as in~\cite{Segev} via the inclusion of such an intermediate poset.
So to get~(Goal), we will look for hypotheses on our~$H$,
which lead to a suitable poset~$X$---allowing us to prove~(D).
Conditions~(C) and~(E) will then follow easily from the naturality of our construction of~$X$ (and hence~$K_0$).

This sequence (C + D + E) of implications that was used for Theorem~2 provides one route---which
we include in summarizing below several possible routes to~(Goal) suggested by our analysis.
Namely we might proceed via any of the following approaches:
\begin{itemize}
\item Prove~(A) (or~(A$^\prime$)):
      e.g.~ by showing that~$a_* : \tilde{H}_*(Y_0) \to \tilde{H}_*(Y)$ is the zero map, while~$\tilde{H}_*(Y) \neq 0$;
      or more generally, showing just that~$a_*$ is not surjective.
\item Prove~(B):
      e.g.~showing~$b_* : \tilde{H}_*(V_0) \to \tilde{H}_* \bigl( \A_p(H) \bigr)$ is the zero map,
      while~$\tilde{H}_* \bigl( \A_p(H) \bigr) \neq 0$;
      or more generally, showing just that~$b_*$ is not surjective.
\item Prove the following sequence:
\begin{itemize}
\item (C): showing that inclusion---based on a natural construction of~$X$ (and hence~$K_0$);
\item (D): showing that~$\psi_* \neq 0$---based on hypotheses on~$H$, giving the construction of~$X$; 
\item (E): showing that~$c_* = 0$---again based on natural constructions related to~$X$.
\end{itemize}
\end{itemize}

\bigskip

In proving our Theorem~\ref{theoremNormalCaseReduction} below, we will in fact follow the sequence~(C + D + E):
In particular, we will prove~(C) and~(E) via an appropriate new choice of~$X$ and~$K_0$ as mentioned above.
The construction of these posets is inspired by the join-construction of the corresponding posets in Segev's work,
which we described in Definitions~\ref{defn:Kasjoin} and~\ref{defn:K0hatK0};
we will need comparatively slight variations on the mechanics of those earlier join-constructions. 
Then we will obtain~(D)---by using %
a component-related technical hypothesis on~$H$---which essentially defines our new~$X$ as an analogous join:
where now the old~$X_i = \A_p \bigl( \Aut_H(L_i) \bigr)$
are replaced by subposets that we call~$\A_i$---which this time
are defined using essentially the~$\A_p \bigl( \Aut_{C_i(H)}(L_i) \bigr)$, for certain {\em sub\/}groups~$C_i(H)$ of~$H$.
Finally, Theorem~\ref{theoremNormalCaseReduction} can be regarded
as an analogue of the ``core'' part of Theorem~2 in~\cite{Segev}: 
Thus rather than assuming hypotheses on the individual~$L_i$, 
we will instead assume that a suitable mapping~$\psi_H$, from~$\A_p(H)$ itself to the new~$X$, is nonzero in homology. 
That is, we design the construction of~$X$ so that we will ``automatically'' get condition~(D).
Thus for this initial result of ours, condition~(D) is the key point.
(In later Corollary~\ref{cor:genTheorem2Segev},
we will give an analogue more closely following the original hypotheses of Theorem~2 in~\cite{Segev}---that is,
where we {\em do\/} assume hypotheses on the individual components~$L_i$.)

\begin{remark}[Good behavior for~$\psi$ in~(D)]
\label{rk:goodbehaviorpsi}
Our setup of implications-analysis in Remark~\ref{rk:summarylogimplic} already involves a number of different maps;
and we will want to be clear about the different behaviors %
that we will want to be establishing for them.
So we quickly review this aspect in our earlier conditions: 
\begin{quote}
(A,A$^{\prime}$,B): \quad We seek a {\em non-surjective\/} map. 
\end{quote}
\begin{quote}
(D): $\phantom{\text{A$^{\prime}$,B}}$  \quad We seek a {\em nonzero\/} map~$\psi_*$. 
\end{quote}
\begin{quote}
(E): $\phantom{\text{A$^{\prime}$,B}}$  \quad We seek a {\em zero\/} map~$c_*$. 
\end{quote}
In fact we will mainly be focusing on~(D), so that our ``usual'' expectation will then be a nonzero map.
So we will sometimes use the informal expression \textit{good behavior}, 
to express the desired outcome of~$\psi_* \neq 0$ in that situation.
\donerk
\end{remark}

\section{The generalized method, leading to elimination results}
\label{sectionMainTool}

From this section on, we begin to provide alternative versions of Segev's results in~\cite{Segev}.
While Segev's Theorem~2 (which we stated as Theorem~\ref{theoremSegevOriginal} in our Introduction)
requires a common behavior for all the components of~$G$,
we show that in general we can focus on the behavior of a single component~$L$---if
we further assume a suitable inductive hypothesis.
Indeed starting in the following Section~\ref{sec:concreteconsqH1etc},
we will use inductive hypotheses such as~(H1) mentioned earlier in our Introduction.

However for our fundamental result Theorem~\ref{theoremNormalCaseReduction} in this section,
we use a more technical kind of inductive hypothesis---somewhat more directly akin to that for Theorem~2 of~\cite{Segev},
(which we discussed in the previous Section~\ref{sectionDiscussionSegev}).
The resemblance is in fact closest to the~$\A_p(H)$-form deduced in Remark~\ref{rk:ApHformhypThm2}
from the original hypothesis of Theorem~2; we had also stated that form as condition~(D) there.
Recall that~(D) required nonzero-ness in homology of the map induced from the poset map~$\psi$,
which takes~$\A_p(H)$ into the formal join~$X$ of the~$\A_p \bigl( \Aut_H(L_i) \bigr)$ for the components~$L_i$.
Similarly, when we now choose as ``$H$'' the kernel on the~$G$-conjugates~$L_i$ of our single component~$L$, 
our hypothesis in~\ref{theoremNormalCaseReduction} will require nonzero-ness in homology
of a map~$\psi_H$ defined on~$\A_p(H)$---into a corresponding new choice of~$X$.
This time the join~$X$ will arise from factors~$\A_i$ which, though still based on the~$\A_p( \Aut_H(L_i) )$ as in Theorem~2 above, 
are constructed in a somewhat more complicated inductive way---a way we can exploit, in verifying the hypothesis for applications. 

We emphasize one further feature of working with just a single component~$L$:
Note that the product~$N := L_1 \cdots L_t$ of the~$G$-conjugates of~$L$ may well be proper in~$F^*(G)$, 
so that we can't assume faithful action on~$N$.
For example, in the discussion of constructions after Definition~\ref{defn:Kasjoin},
we were able to use faithfulness to embed~$H$ in the formal direct product of the~$\Aut_H(L_i)$;
but this time such an embedding holds only for the faithful quotient~$H/C_H(N)$.  
For this reason, our new join~$X$ in the statement of~\ref{theoremNormalCaseReduction}
also includes a term~$\A_p \bigl( C_H(N) \bigr)$; which later we will similarly call~$\A_0$.

\bigskip

After the proof of Theorem~\ref{theoremNormalCaseReduction}, 
the section proceeds to Proposition~\ref{propositionPropertiesEM}---which 
gives one possible approach to verifying the hypotheses of~\ref{theoremNormalCaseReduction}
(though we won't actually be quoting Proposition~\ref{propositionPropertiesEM} in the later sections of this paper).
That approach more explicitly involves the filtration of~$H$ given by certain centralizer-subgroups~$C_i(H)$ below,
and the building-up of the map~$\psi_H$ via successive approximations~$\psi_i$. 
Finally, we will close the section with Theorem~\ref{theorem1},
which provides a generalization of Theorem~1 of~\cite{Segev}.

\bigskip

Thus we begin the work of the section by stating our main technical result below:

\begin{theorem}
\label{theoremNormalCaseReduction}
Let~$L\leq G$ be a component of order divisible by~$p$, with~$\{ L_1 , \ldots , L_t \}$ its~$G$-orbit under conjugation.
Set~$H := \bigcap_i N_G(L_i)$ and~$N := L_1 \ldots L_t$, with posets~$\A_i$ as in Definition ~\ref{defn:CalAi},
and poset map~$\psi_H$ as in Definition~\ref{defn:psiH}.
Assume also that:
\begin{center}
  $\psi_H : \A_p(H)  \to X := \A_p \bigl( C_H(N) \bigr) *  \A_1 * \ldots * \A_t$ satisfies~$(\psi_H)_* \neq 0$ in homology.
\end{center}
(Indeed, it suffices if there exists a subposet~$\B \subseteq \A_p(H)$ with~$(\psi_H |_{\B})_* \neq 0$.)

Then~$G$ satisfies~(H-QC).

\smallskip

\noindent
Hence such an~$L$ is eliminated from a counterexample to~(H-QC).
\end{theorem}

We temporarily postpone the proof, while we provide various relevant background details.
 
\bigskip

\noindent
{\em Preliminaries: construction of posets and maps for Theorem~\ref{theoremNormalCaseReduction}\/}

\bigskip

\noindent
We begin with a few further comments relevant to hypotheses:

The results in~\cite{Segev} assume the hypothesis~$O_{p'}(G) = 1$.
However, in our Theorem~\ref{theoremNormalCaseReduction} here,
this can be relaxed to only requiring that some components $L$ of~$G$ has order divisible by~$p$;
as we will emphasize in the proof below.
In particular, in paralleling arguments from~\cite{Segev},
we must be careful about uses there of consequences of~$O_{p'}(G) = 1$ in Remark~\ref{rk:conseqOp'G=1}
(especially since these are sometimes not explicit in~\cite{Segev}). 
In the case of~\ref{rk:conseqOp'G=1}(i), properties related to faithful action 
can be handled by careful treatment of centralizers which may now be nontrivial.
For example, our component~$L$ might not now be simple---but since we assume~$O_p(G) = 1$ in results on~(H-QC), 
we will at least have the following partial replacement for~\ref{rk:conseqOp'G=1}(i):

  (i)$_L$   \quad $Z(L)$ is a~$p'$-group.

\noindent
On the other hand, \ref{rk:conseqOp'G=1}(ii) is now directly assumed in our hypothesis. 
So we will call that:

  (ii)$_L$  \quad $p$ divides the order of the component~$L$.

\noindent
And then we immediately obtain the consequence~\ref{rk:conseqOp'G=1}(iii)---which we will name similarly:

  (iii)$_L$ \quad $\A_p(L) \neq \emptyset$, and hence~$\A_p(H) \neq \emptyset$;

\noindent
as before this will again be fundamental. 

Furthermore, later in the paper we will wish to work under hypothesis~(H1): 
and then we will get by Remark~\ref{remarkH1andHLp} that either~(H-QC) holds for~$G$; or else~$O_{p'}(G) = 1 = O_p(G)$---in which case
by Lemma~\ref{lemmaOpandp}, the components of~$G$ have order divisible by~$p$.
That is, we will then get the condition~(ii)$_L$---giving us that hypothesis,
when we want to set up to apply Theorem~\ref{theoremNormalCaseReduction}. 

%Now we briefly explain the terminology of $C_i(H)$ and $\A_i$ of the above theorem.
%Let $N = L_1\ldots L_t$, and for $i\geq 0$ set $C_i(H):=C_H(L_{i+1}\ldots L_t)$, which is a subgroup of $H$.
%Set $\A_0 := \A_p(C_H(N))$.
%For $i>0$, the poset $\A_i$ is the image of the map $\pi_i:\A_p(C_i(H)) - \A_p(C_{i-1}(H)) \to \A_p(\Aut_{C_i(H)}(L_i))$
%induced by taking the quotient by the centralizer $C_{C_i(H)}(L_i) = C_{i-1}(H)$ (see Definition \ref{definitionCalAi}).
%The map $\psi_H:\A_p(H)\to \A_0*\A_1*\ldots*\A_t$ sends an element $A\in\A_p(H)$ to $\pi_i(A)\in\A_i$,
%where $i$ is the largest index such that $A\in \A_p(C_i(H)) - \A_p(C_{i-1}(H))$.
%Note that $C_i(H)$ denotes a centralizer like subgroup of $H$ and it is not the chain complex group of any poset.
%{\color{red} Since we don't use chain complexes here, is this comment necessary?}

\bigskip

Now we introduce some further terminology;
including the definition of the~$\A_i$ needed in the statement of Theorem~\ref{theoremNormalCaseReduction} above. 
We will use these concepts throughout this and the remaining sections;
the image-posets~$\Imageposet{G}{L}$ and~$\A_i$ will play a fundamental role
in our constructions and arguments in the proof.

\begin{definition}
\label{defn:CalA}
For any subgroup~$T$ of~$G$, we view the usual projection map just on~$\A_p$-posets:
  \[ \pi_T : \A_p \bigl( N_G(T) \bigr)\  \to\ \A_p \bigl( \Aut_G(T) \bigr) \cup \{ 1 \} . \]
Define the \textit{image poset}~$\Imageposet{G}{T}$ as the restriction of~$\pi_T$
to~$ \A_p \bigl( N_G(T) \bigr) - \A_p \bigl( C_G(T) \bigr)$. 
Equivalently:
  \[   \Imageposet{G}{T} := \Im[\ \pi_T :\ \A_p \bigl( N_G(T) \bigr)\ \to\ \A_p \bigl( \Aut_G(T) \bigr) \cup \{ 1 \}\ ]\ -\ \{ 1 \} . \]
\end{definition}

\noindent
In Lemma~\ref{lemmaPOuterPreimage},
we will give some alternative descriptions of the image-poset~$\Imageposet{G}{T}$,
in the case where~$T$ is any subgroup of~$G$ with~$p$ not dividing the order of~$Z(T)$.
In addition, if $T$ is quasisimple,
we will show that~$\tilde{H}_*( \Imageposet{G}{T} ) \neq 0$ (see Theorem~\ref{thm:HomologyCalA}).
For later applications when~$T$ is a component~$L$ of~$G$, 
this will be relevant to establishing the hypothesis of Theorem~\ref{theoremNormalCaseReduction}.

Next we define the posets~$\A_i$, based on corresponding centralizer-subgroups~$C_i(H)$.
We note that the definition of these objects depends on
the choice of the ordering of our~$G$-orbit of components~$L_1,\ldots,L_t$ of~$G$;
and the same holds for the definition of the subgroups~$C_i(H)$ and maps~$\psi_{C_i(H)}$ shortly thereafter.
But this will not be an issue---cf.~our comment after Definition~\ref{defn:iEandpsi}:
namely we will only need the fact that the~$\psi_{C_i(H)}$ are poset maps---which holds for any choice of that ordering.

\begin{definition}
\label{defn:CalAi}
Let $L_1,\ldots, L_t$ be a $G$-orbit of components of $G$.
We define:
\begin{itemize}
\item $H := \bigcap_i N_G(L_i)$ and~$N := L_1 \ldots L_t \leq H$.
\item For~$0 \leq i \leq t$, let $C_i(H) := C_H(L_{i+1} \ldots L_t)$, with~$C_t(H) := C_H(1) = H$. \\
   Note~that $L_i \leq C_i(H)$. 
\item We will use the abbreviation~$\A_0 := \A_p \bigl( C_H(N) \bigr)$, \\
   and write~$\pi_0$ for the identity map~$\A_p \bigl( C_H(N) \bigr) \to \A_p \bigl( C_H(N) \bigr)$.
\item For $i \geq 1$, let~$\pi_i$ be the map:
   \[ \pi_i : \A_p \bigl( C_i(H) \bigr) \to \A_p \bigl( \Aut_{C_i(H)}(L_i) \bigr) \cup \{ 1 \} , \]
  which is induced via taking the quotient by~$C_{C_i(H)}(L_i) = C_{i-1}(H)$.
\item For~$i \geq 1$, the image-poset~$\A_i$ is the poset:
   \[ \A_i := \Imageposet{C_i(H)}{L_i} . \]
  Using~(i)$_L$ for faithfulness in Lemma~\ref{lemmaPOuterPreimage}, we may write~$\A_p(L_i) \subseteq \A_i$. \\
  Note also that an element~$E$ of~$\A_i$ is the quotient
  of an element of~$\A_p(C_i(H))$ by~$C_{C_i(H)}(L_i)$;
  so we have~$\A_i \subseteq \A_p(\Aut_{C_i(H)}(L_i)) \subseteq \A_p(\Aut_H(L_i))$.%
\footnote{ 
  Further the restriction~$\pi_i : \A_p \bigl( C_i(H) \bigr) - \A_p \bigl( C_{C_i(H)}(L_i) \bigr) \to \A_i$
  gives a quotient map in the topological sense.
          }
\end{itemize}
\end{definition}

\noindent
For the next few paragraphs, we explore the general properties of the objects defined above.
So we fix a component~$L$ of~$G$, and an ordering of its~$G$-orbit~$L_1,\ldots, L_t$.
We may take~$L_t := L$.
Finally let~$H$, $N$, $C_i(H)$ and~$\A_i$ be as in the definitions above.

\begin{remark}
\label{remarkDescriptionAiCi}
Note that:
  \[ C_0(H) = C_H(L_1 \ldots L_t) = C_H(N) = C_G(N) , \]
and by convention:
  \[ C_t(H) = C_H(1) =  H . \]
Moreover, we have a normal series of~$H$ given by:
  \[ C_H(N) = C_0(H) \normal C_1(H) \normal \ldots \normal C_t(H) = H . \]
We also have that~$H$, $N$ and~$C_H(N)$ are normal subgroups of~$G$.

Further~$L_i \leq C_i(H)$ and~$C_{C_i(H)}(L_i) = C_{i-1}(H)$;
so~$\Aut_{C_i(H)}(L_i) = C_i(H)/C_{i-1}(H)$,
and~$\A_i$ is the image of~$\pi_i$ restricted to~$\A_p \bigl( C_i(H) \bigr) - \A_p \bigl( C_{i-1}(H) \bigr)$,
since~$\pi_i^{-1}(1) = \A_p \bigl( C_{i-1}(H) \bigr)$.
\donerk
\end{remark}

\bigskip

With the above properties established, we will now able to construct the map~$\psi_H$---still
needed for the statement of Theorem~\ref{theoremNormalCaseReduction}.
This will be a variant of Segev's map~$\psi : \A_p(H) \to X$,
which we had described in Definition~\ref{defn:iEandpsi}. 

For the purposes just of proving Theorem~\ref{theoremNormalCaseReduction},
it would be sufficient to now simply define~$\psi_H$ essentially via the condition in Definition~\ref{defn:iEandpsi}:
namely for~$E \in \A_p(H)$, as the projection of~$E$ on~$\Aut_H(L_{i_E})$, 
where~$i_E$ is the largest index~$j$ for which~$E \nleq C_H(L_j)$. 
However, for the purposes of verifying the $\A_i$-related hypotheses of that theorem in later applications,
it will be convenient to see below that we can ``inductively'' build up~$\psi_H$ from successive approximations~$\psi_i$,
using the filtration of~$H$ given by the~$C_i(H)$;
the definition via~$i_E$ then emerges naturally from this process:

\begin{definition}[The maps~$\psi_i$ and~$\psi_H$]
\label{defn:psiH}
For~$0 \leq i \leq t$, set $W_i := \A_0 * \A_1 * \ldots * \A_i$; note~$X:=W_t$.
Define~$\psi_{C_i(H)} : \A_p \bigl( C_i(H) ) \to W_i$ on~$E \in \A_p \bigl( C_i(H) \bigr)$ by: 
\begin{align*}
\psi_{C_i(H)}(E) & :=           \pi_k(E) ,\quad k = \max \{\ j \leq i \tq E \nleq C_H(L_j)\ \} \\
                 & \phantom{:}= \pi_k(E) ,\quad k = \min \{\ j \leq i \tq E \leq C_j(H)\ \} .
\end{align*}
Notice since~$E \leq C_i(H)$ that~$E$ centralizes~$L_j$ for all~$j >i$; 
hence in the top line above, we could remove the restriction~``$j \leq i$''; 
that is, the indicated ~``$k$'' is indeed the index that we called~$i_E$ in Definition~\ref{defn:iEandpsi}:
namely  the largest index~$j$ for which~$E \nleq C_H(L_j)$.
In particular, we see that~$\psi_{C_i(H)}(E) = \pi_{i_E}(E)$ lies in the image-poset~$\A_{i_E}$ (and in no other~$\A_j$),
and hence (since~$i_E \leq i$) in~$W_i$. 

We had seen that~$\A_p(L_i) \subseteq \A_i$; this says for~$E \in \A_p(L_i)$ that~$i_E = i$. 
Thus for the restriction of~$\psi_{C_i(H)}$ to~$\A_p(L_i)$, the usual projection becomes just the natural ``inclusion''.
This is the analogue of the property ``$\psi |_{\A_p(E(G))} = \pi$'' that we saw before Remark~\ref{rk:ApHformhypThm2}.

\textbf{Notation.} We usually write~$\psi_H$, or just~$\psi$, to abbreviate~$\psi_{C_t(H)}$; and~$\psi_i$ for~$\psi_{C_i(H)}$.
\donerk
\end{definition}

\smallskip

The following is an easy consequence of the definition above and Remark~\ref{remarkDescriptionAiCi}:

\begin{lemma}
\label{lemmaContainmentsCiHWi}
The maps~$\psi_{C_i(H)}$ are order-preserving. 
Moreover, for all~$0 \leq i \leq k \leq t$, we have a commutative diagram:
\[ \xymatrix{
     \A_p \bigl( C_i(H) \bigr) \ar@{^(->}[d] \ar[rr]^{\psi_{C_i(H)}} & & W_i \ar@{^(->}[d] \\
     \A_p \bigl( C_k(H) \bigr) \ar[rr]^{\psi_{C_k(H)}}                             & & W_k
             }
\]
\end{lemma}

\bigskip

At this point, much as in our discussion after Definition~\ref{defn:Kasjoin}, 
we see that~$\psi_H$ maps~$\A_p(H)$ to the formal join~$X$---of posets~$\A_i$,
related to the individual~$\A_p \bigl( \Aut_H(L_i) \bigr)$, for~$L_i$ now just in the orbit of~$L$;
and the ``trivial-action part''~$\A_0$ containing the kernel~$C_H(N)$ of the action of~$H$ on each member of the orbit of~$L$.
We have now completed our preparations for:

\bigskip

\noindent
{\em The main argument for Theorem~\ref{theoremNormalCaseReduction}\/}

\bigskip

Thus we can finally prove our main technical result Theorem~\ref{theoremNormalCaseReduction}.
In view of our preliminary discussion of its hypothesis,
the proof will parallel the part of the proof of Segev's Theorem~2 in Section~\ref{sectionDiscussionSegev}, 
{\em after\/} we obtained the~$\A_p(H)$-form of the hypothesis in Remark~\ref{rk:ApHformhypThm2}:

\begin{proof}[Proof of Theorem~\ref{theoremNormalCaseReduction}]
As usual in proving~(H-QC), we assume that~$O_p(G) = 1$.
We show that, under the hypotheses of this theorem, $\tilde{H}_* \bigl( \A_p(G) \bigr) \neq 0$;
recall we had earlier called this~``(Goal)''.

To that end,
we will follow the outline of Segev's argument in our implications-analysis in Remark~\ref{rk:summarylogimplic}.
In brief summary:
We will prove that conditions~(C), (D) and~(E) hold.
We will see that~(C) and~(E) follow from our definition of the appropriate new~$X$ (and of the analogous~$K_0$).
And our hypothesis, of nonzero-ness in homology of the natural poset map~$\psi_H : \A_p(H) \to X$,
was designed to give exactly condition~(D).

\bigskip

\noindent
We now begin the detailed proof:

We recall that as~``$H$'',
we are here using the {\em local\/} kernel---namely the kernel of~$G$ on the orbit of~$L$.
Also we have a new~``$X$'', defined as the join of the~$\A_i$ (see Definition~\ref{defn:CalAi}).
With these changes, our basic setup will be as in~\cite{Segev}.

Just as in our earlier analysis of~\cite{Segev},
we need to first reduce to the subcase~$\A_p(H) \subsetneq \A_p(G)$. 
This closely follows our earlier discussion of the proof of Segev's Theorem~2:
Note that in any case for~$\A_p(H)$, our hypothesis of~$(\psi_H)_* \neq 0$ 
forces~$\tilde{H}_* \bigl( \A_p(H) \bigr) \neq 0$ for its domain-space.
So when~$\A_p(H) = \A_p(G)$, we in fact have the nonzero homology of~$\A_p(G)$ required for~(Goal).
Thus for this case, the hypothesis already {\em includes\/} the conclusion~(H-QC) of the result---as required.

With this reduction to~$\A_p(H) \subsetneq \A_p(G)$ (so that~$H < G$) in hand,
we again obtain consequences as in earlier Remark~\ref{rk:conseqH<G}---though now we use~(iii)$_L$ above 
to get~$\A_p(H) \neq \emptyset$, in place of the use there of~\ref{rk:conseqOp'G=1}(iii). 
So in analogy with~\ref{rk:conseqH<G}(ii) we get:

\centerline{
     We have~$t \geq p \geq 2$ components~$L_i$ in our~$G$-orbit for~$L$.
            }

\noindent
Furthermore from~\ref{rk:conseqH<G}(i) we get nontriviality of the fundamental decomposition for~$\A_p(G)$: 
Let's in fact review that basic setup of~\cite{Segev} from the discussion around Remark~\ref{rk:conseqH<G}.
Set~$Y := \N_G(H)$, with~$Z := \A_p(G) - \A_p(H)$, and~$Y_0 := Y \cap Z = \N_G(H) - \A_p(H)$.
We recall also the earlier retraction~$r : Y \to \A_p(H)$, and set~$V_0 := r(Y_0)$. 
Then~\ref{rk:conseqH<G}(i) gives us: 

\centerline{
  The decomposition~$\A_p(G) = Y \cup Z$ is nontrivial; with $Y,Z,Y_0,V_0 \neq \emptyset$. 
            }

\noindent
Then from Remark~\ref{rk:MVseqfordecompYZ} we have the corresponding Mayer-Vietoris exact sequence:
 \[ \ldots \to \tilde{H}_{k+1} \bigl( \A_p(G) \bigr)
           \to \tilde{H}_k(Y_0)
	   \to \tilde{H}_k(Y) \oplus \tilde{H}_k(Z)
	   \to \tilde{H}_k \bigl( \A_p(G) \bigr)
	   \to \ldots .
\]

\noindent
As before, we will get the desired conclusion~$\tilde{H}_*(\A_p(G)) \neq 0$,
if the map~$a_*$ induced in homology by the inclusion~$a : Y_0 \subseteq Y$ is not surjective. 
(Cf.~condition~(A) in Remark~\ref{rk:initimpl}.)

We saw by Remark~\ref{remarkPropertiesOfCalN}
that the retraction~$r : Y \to \A_p(H)$, given by~$r(E) = E \cap H$, is a homotopy equivalence. 
The analogue of the earlier commutative diagram in Remark~\ref{rk:diagramSegev}
is given here by diagram~(\ref{diagramProofMainThm}) below;
and we now see similarly that we get our conclusion~$\tilde{H}_* \bigl(\A_p(G) \bigr) \neq 0$,
if in fact the~$r$-translated map~$b_*$ induced by the inclusion~$b : V_0 \subseteq \A_p(H)$ is not surjective.
(Cf.~condition~(B) in Remark~\ref{rk:initimpl}.)

We now make appropriate adjustments to the earlier constructions of Segev.
Recall we have already defined some new variants,
of the older constructions given at Definitions~\ref{defn:Kasjoin} and~\ref{defn:iEandpsi}:

\centerline{
  $X = \A_0 * \A_1 * \ldots * \A_t$, with poset map~$\psi_H : \A_p(H) \to X$. 
            }

\centerline{
In the case where~$\A_0 = \emptyset$, we simply take~$X = \A_1 * \ldots * \A_t$.
            }
\noindent
In particular we saw, in defining~$\psi_H$ as~$\psi_{C_t(H)}$ at Definition~\ref{defn:psiH},
that for any~$E \in \A_p(H)$,
taking~$i_E$ to be the largest index~$j$ with~$E \nleq C_H(L_j)$, we have~$\psi_H(E) = \pi_{i_E}(E) \in \A_{i_E}$.
Furthermore for our new choice of~$X$ above, we define the subcomplex~$K_0$ 
in parallel with earlier Definition~\ref{defn:K0hatK0}:
Namely we write~$\A_{\hat{i}}$ for the maximal sub-join given by the~$\A_j$ for~$j \neq i$; and then take:

\centerline{
  $K_0 := \bigcup_{i=0}^t\ \K( \A_{\hat{i}})$ (but exclude $i=0$ when $\A_0 = \emptyset$).
            }

\noindent
This definition means that every chain~$\sigma$ from~$K_0$ has no contribution from some~$\A_{i_{\sigma}}$,
i.e.~for at least one index~$i_{\sigma}$.
Recall that our initial reduction to the nontrivial-action case~$H < G$
guarantees~$t \geq p \geq 2$ for the number of components in our orbit;
recall we saw at Definition~\ref{defn:K0hatK0} that this guarantees that~$K_0 \neq \emptyset$.

These adjusted-constructions provide the remaining terms needed for our desired overall commutative diagram,
in analogy with that in earlier Remark~\ref{rk:diagramSegev}:

\begin{equation}\label{diagramProofMainThm}
\xymatrix{
  \tilde{H}_n(Y_0) \ar[d]_{a_*} \ar[r]^{({r|_{Y_0})}_*} & \tilde{H}_n(V_0) \ar[d]_{b_*} \ar[r]^{{(\psi_H|_{V_0})}_*}
                                                        & \tilde{H}_n(K_0) \ar[d]^{=0}_{c_*} \\
  \tilde{H}_n(Y) \ar[r]^{r_*}_{\cong}                  & \tilde{H}_n(\A_p(H)) \ar[r]^{{(\psi_H)}_*}_{ \neq 0}
                                                        & \tilde{H}_n(X)
          }
\end{equation}

\noindent
That is: to verify the properties in the diagram, and hence complete our proof, it suffices
(in view of the discussion of implications in Remark~\ref{rk:summarylogimplic})
to establish conditions~(C), (D) and~(E).

\bigskip

We show first that~$\psi_H$ maps~$\K(V_0)$ into~$K_0$:

\noindent
\textbf{Claim:} (C) holds; that is, $\psi_H \bigl( \K(V_0) \bigr) \subseteq K_0$.

\noindent
Here we have a partial analogue of the proof of Segev's inclusion~$j : V_0 \subseteq \D_p(H)$
(for his Theorem~1, as we described in Section~\ref{sectionDiscussionSegev}):
We saw that our earlier reduction to~$H < G$ gives~$V_0 \neq \emptyset$. 
So we consider any~$\sigma \in \K(V_0)$; and we need to get~$\psi_H(\sigma) \in K_0$. 
We can take~$\sigma = (A_0 < \ldots < A_l)$,
where there is~$E \in \A_p(G)$ with~$E \nleq H$ and~$A_l = E \cap H$.
So each~$A_k \leq A_l \leq E$.
Recall that~$H$ is the local-kernel: 
so the elements~$e \in  E - H$ induce nontrivial permutations on the~$G$-orbit~$\{ L_1,\ldots, L_t \}$.
Let~$i_1 < i_2 < \dots < i_p$ (with~$p \geq 2$) be the indices for some such nontrivial~$e$-orbit.
We will show that each~$\psi_H(A_k)$ lies in~$\A_{\widehat{i_1}}$:
Recall we have~$\psi_H(A_k) = \pi_{i_{A_k}}(A_k)$---where~$i_{A_k}$ is the largest index~$j$ with~$A_k \nleq C_H(L_j)$.
We claim that~$i_{A_k} \neq i_1$:
If~$A_k \leq C_H(L_{i_1})$, then~$i_{A_k} \neq i_1$ because we just saw that~$A_k \nleq C_H(L_{i_{A_k}})$. 
Otherwise~$A_k \nleq C_H(L_{i_1})$;
but then as~$A_k$ centralizes~$e$,  $A_k$ is also nontrivial on the remaining members of the~$e$-orbit,
including~$L_{i_p}$ with~$i_p > i_1$; 
so~$i_1$ is not the {\em largest\/} index~$j$ with~$A_k \nleq C_H(L_j)$, as required above.
Thus we have shown that~$i_{A_k} \neq i_1$, and hence~$\psi_H(A_k) \not \in \A_{i_1}$---in fact for all~$k$;
so that~$\psi_H(\sigma) \in \K ( \A_{\widehat{i_1}} ) \subseteq K_0$, as needed.

\bigskip

Condition~(D), i.e.~$(\psi_H)_* \neq 0$, holds since it is already part of the hypothesis;
indeed this was the main motivation for the design of the hypothesis.

\vspace{0.2cm}

Hence it remains to establish condition~(E).

\noindent
\textbf{Claim:} (E) holds, that is, $c : K_0 \subseteq \K(X)$ induces the zero map in homology.

\noindent
We had already included essentially this observation in Definition~\ref{defn:K0hatK0}; 
for convenience we review the argument here:
Recall we followed~\cite[p958]{Segev} in describing the construction of a standard contractible simplicial complex~$\hat{K}_0$,
such that~$K_0 \subseteq \hat{K}_0 \subseteq \K(X)$;
namely we gave the definition:

\centerline{
   $\hat{K}_0 := \bigcup_i\ \St_{\K(X)}(v_i)$,
            }

\noindent
where the~$v_i \in \A_i$ are fixed vertices, and~$\St_{\K(X)}(v_i)$ is the star of~$v_i$ in~$\K(X)$.
This complex is contractible (cf.~\cite[(5.1)]{AS92}).
Notice that at this point,
we are using that~$\A_i \neq \emptyset$ for all~$i$, so that these vertices~$v_i$ exist.
So to see that~$\A_i \neq \emptyset$:
For~$i \geq 1$, 
we recall that~$p$ divides the order of each~$L_i$ by~(ii)$_L$, while~$Z(L_i)$ is a~$p'$-group by~(i)$_L$;
so we can use Lemma~\ref{lemmaPOuterPreimage} to see that~$\A_i \supseteq \A_p(L_i) \neq \emptyset$.
For~$i=0$, if~$\A_0 = \emptyset$, then we had excluded it from the definition of~$X$ and hence of~$K_0$.
(If we had~$\A_i = \emptyset$ for some~$i$, then~$K_0 = \K(X)$;
and the claim would not hold unless~$\tilde{H}_*(X) = 0$---in which case we would not be using this~$X$.)

\bigskip

We now complete the proof of Theorem~\ref{theoremNormalCaseReduction}:
From~(\ref{diagramProofMainThm}), $0 = c_* \circ (\psi_H |_{V_0})_* = (\psi_H)_* \circ b_*$;
so as~$(\psi_H)_* \neq 0$, $b_*$ cannot be surjective. 
From our earlier analogue of~(B), we get~$\tilde{H}_* \bigl( \A_p(G) \bigr) \neq 0$---that is, (Goal). 

For the ``it suffices'' part of the statement of the theorem, 
note that~$\psi_H |_{\B} = \psi_H \circ f$, for the inclusion~$f : \B \subseteq \A_p(H)$;
so that $(\psi_H |_{\B})* \neq 0$ forces $(\psi_H)_* \neq 0$. 
\end{proof}

In the remainder of the section, we present:

\bigskip

\noindent
{\em Further technical results via the generalized Segev method\/}

\bigskip

We now give some definitions and technical properties that can be useful
for establishing, in applications, the hypotheses of Theorem~\ref{theoremNormalCaseReduction}.
These are inspired by the filtration of~$\A_p(H)$
determined by the~subposets~$\A_p \bigl( C_i(H) \bigr)$ of sub{\em groups\/} of~$H$ (from Definition~\ref{defn:CalAi}):
which we use now for approaching the posets~$\A_i$ (of~$\A_p$-posets of {\em quotient\/} groups of the~$C_i(H)$)
appearing as the factors in the join~$X$---where of course~$X$ is the codomain-space for our crucial mapping~$\psi_H$.
Indeed we could roughly  view the process of such successive subgroup-by-quotient replacements
as gradually transforming the filtration of~$\A_p(H)$ via the~$\A_p \bigl( C_i(H) \bigr)$,
into our filtration of~$X$ determined by the sub-joins~$W_i$ (from the~$\A_j$ for~$j \leq i$).

So below we consider maps taking the~$i$-th subgroup-term~$\A_p \bigl( C_i(H) \bigr)$
to a version of itself---in which we roughly ``shrink the top part'' to a join-factor,
given by the poset of quotients of its members appearing in~$\A_i$. 
We can then view our earlier map~$\psi_i$ as being defined downward (in the ordering on components)
from the~$i$-th stage of these shrinkings;
and then our final map~$\psi_H$ arises as we proceed upward through the successive ``approximations''~$\psi_i$. 
Thus our filtration-transforming process inductively builds toward our goal of homology propagation from~$\A_p(H)$ to~$\A_p(G)$:
and along the way we will be propagating good behavior in homology (in the sense of Remark~\ref{rk:goodbehaviorpsi})
for the intermediate maps~$\psi_i$ leading to our final~$\psi_t = \psi_H$.

Hence for~$1 \leq i \leq j \leq t$, we define the following maps:
\begin{align}
 \varphi_i    & : \A_p \bigl( C_i(H) ) \to \A_p \bigl( C_{i-1}(H) \bigr) * \A_i \\
 \varphi_i(E) & := \begin{cases}
                    E \in \A_p \bigl( C_{i-1}(H) \bigr) & \text{if } E\leq C_{i-1}(H), \\
                    \pi_i(E) \in \A_i                   & \text{if } E\nleq C_{i-1}(H).
\end{cases}
\end{align}

\begin{align}
\Phi_{i,j}    & : \A_p \bigl( C_i(H) \bigr) * \A_{i+1} * \ldots * \A_j
                  \to \A_p \bigl( C_{i-1}(H) \bigr) * \A_{i} * \A_{i+1} * \ldots * \A_j\\
\Phi_{i,j}(E) & := \begin{cases}
                    \varphi_i(E) \in \A_p \bigl( C_{i-1}(H) \bigr) * \A_{i} & \text{if } E \in \A_p \bigl( C_{i}(H) \bigr) , \\
                    E \in \A_k                                              & \text{if } E \in \A_k,\ i+1 \leq k \leq j.
\end{cases}
\end{align}

\noindent
The following lemma summarizes the relations between the maps $\Phi$, $\varphi$ and $\psi$:

\bigskip

\begin{lemma}
[{Relations $\Phi-\psi$}]
\label{lemmaRelationPhiPsi}
The following hold:
\begin{enumerate}
 \item $\Phi_{i,j} = \varphi_i * \Id_{\A_{i+1}} * \ldots * \Id_{\A_j}$.
 \item $\Phi_{i,i} = \varphi_i$.
 \item $\psi_i = \Phi_{1,i} \circ \ldots \circ \Phi_{i-1,i} \circ \Phi_{i,i}$.
 \item If~$\varphi_i$ induces an epimorphism in all the homology groups of degree~$\leq n$, \\
         then so does~$\Phi_{i,j}$ in all the homology groups of degree~$\leq n + (j-i)$, for all~$j \geq i$.
 \item If~$\varphi_i$ induces a monomorphism in all the homology groups of degree~$\leq n$, \\
         then so does~$\Phi_{i,j}$ in all the homology groups of degree~$\leq n + (j-i)$, for all $j \geq i$.
\end{enumerate}
\end{lemma}

\begin{proof}
Parts~(1), (2) and~(3) are straightforward.
For~(4) and~(5), note that the tensor product is exact over the field~$\QQ$,
and that~$\Phi_{i,j} = \varphi_i * \Id_{\A_{i+1}} * \ldots * \Id_{\A_j}$ induces in homology
the tensor product map~$(\varphi_i)_* \otimes (\Id_{\A_{i+1}})_* \otimes \ldots \otimes (\Id_{\A_j})_*$
(see equation~\ref{equationJoinHomology}).
%We are also using the fact that $\A_i\neq\emptyset$ if $i\geq 1$ (see Remark \ref{descriptionAi}).
\end{proof}

We consider now the two properties below,
which will encode sufficient conditions to propagate homology from the posets~$\A_p(C_i(H))$ to~$\A_p(H)$.
Then in Proposition~\ref{propositionPropertiesEM},
we will indicate one way of guaranteeing the conditions of Theorem~\ref{theoremNormalCaseReduction} from these properties.
Later we will indicate still other ways of using these properties.

We begin by describing the common setup for both properties:  
\begin{definition}[The map~$\varphi$ for a normal~$L$]
\label{defn:phiforL}
Assume~$L \normal G$ is a normal component of~$G$;
and denote by $\pi : \A_p(G) \to \A_p \bigl( \Aut_G(L)\bigr) \cup \{ 1 \}$ the map induced by the quotient.
Set~$\A := \Imageposet{G}{L}$; and define $\varphi :\A_p(G) \to \A_p \bigl( C_G(L) \bigr) * \A$ by:
 \[ \varphi(E) = \begin{cases}
                   E \in \A_p \bigl( C_G(L) \bigr) & \text{if } E \leq C_G(L), \\
                   \pi(E) \in \A                   & \text{if } E \nleq C_G(L).
                 \end{cases}
 \]
 We will apply this to successive components~$L_i$ via our filtration.
\donerk
\end{definition}

In this context, we define first the epimorphism-property:

\smallskip

\noindent
\textbf{Property E(n).} 
We have an~$n \geq 0$, for which in the setup of Definition~\ref{defn:phiforL}, we get: 
\begin{center}
  $ \varphi_* : \tilde{H}_m \bigl( \A_p(G) \bigr) \to \tilde{H}_m (\ \A_p \bigl( C_G(L) \bigr) * \A\ )$
   is an epimorphism for all~$m \leq n$.
\end{center}

\smallskip

\noindent
We then define analogously the monomorphism-property:

\smallskip

\noindent
\textbf{Property M(n).}
We have an~$n \geq 0$, for which in the setup of Definition~\ref{defn:phiforL}, we get:
\begin{center}
  $\varphi_* : \tilde{H}_m \bigl( \A_p(G) \bigr) \to \tilde{H}_m (\ \A_p \bigl( C_G(L) \bigr) * \A\ )$
   is a monomorphism for all~$m \leq n$.
\end{center}

\vspace{0.3cm}

\noindent
The hypotheses of Proposition~\ref{propositionPropertiesEM} below
give one possible approach to conditions which are sufficient,
in terms of the above Properties~E($n$) and~M($n$),
to establish the hypothesis of Theorem~\ref{theoremNormalCaseReduction}.
For~\ref{propositionPropertiesEM}, we will not assume the inductive hypotheses~(H1) or~(H$L(p)$) mentioned earlier
(though we {\em will\/} later assume them in some other results which involve the properties).
Instead we will assume a special inductive assumption on the homology of certain relevant~$p$-subgroup posets.
Note also that~$n + t -i \geq 0$ below: that is, $i \leq n +  t$---since~$i \leq t$ and~$n \geq 0$.

\begin{proposition}
\label{propositionPropertiesEM}
Let~$L$ be a component of~$G$ of order divisible by~$p$,
and let~$\{ L_1 ,\ldots , L_t \}$ denote its~$G$-orbit under conjugation.
Let~$H := \bigcap_i\ N_G(L_i)$ and~$N := L_1 \cdots L_t$.
Suppose that there exists~$n \geq 0$ such that one of the following holds:
\begin{enumerate}
 \item $\tilde{H}_n \bigl( \A_p(H) \bigr) \neq 0$,\
         and for each~$1 \leq i \leq t$, $\bigl( L_i, C_i(H) \bigr)$ has Property~M($n-t+i$).
 \item $\tilde{H}_n(\ \A_p \bigl( C_H(N) \bigr) * \A_1 * \ldots *\A_t\ ) \neq 0$, \\
         and for each~$1 \leq i \leq t$, $\bigl( L_i, C_i(H) \bigr)$ has Property~E($n-t+i$).
\end{enumerate}
Then the hypotheses of Theorem~\ref{theoremNormalCaseReduction} hold, so that~$G$ satisfies~(H-QC).
\end{proposition}

\begin{proof}
We will show that~(1) (resp.~(2)) fulfills the hypotheses of Theorem~\ref{theoremNormalCaseReduction}.
By Lemma~\ref{lemmaRelationPhiPsi}(3) with~$t$ in the role of~``$i$'' there,
we have~$\psi_H = \psi_t = \Phi_{1,t} \circ \ldots \circ \Phi_{t,t}$.
We now fix some index~$i \geq 1$; and consider the term~$\Phi_{i,t}$ in the above composition.
Note first that~$C_{C_{i}(H)}(L_i) = C_{i-1}(H)$:
so we can use~$L_i$, $C_i(H)$ in the roles of~``$L$, $G$'' in the setup of Definition~\ref{defn:phiforL},
to see that the role of~``$\varphi$'' there is played here by the map~$\varphi_i$ defined just before Lemma~\ref{lemmaRelationPhiPsi}. 
Now by hypothesis we have Property~M($n+t-i$) (resp.~Property~E($n+t-i$)) for the pair~$\big( L_i, C_i(H) \bigr)$.
This means that~$\varphi_i : \A_p \bigl( C_i(H) \bigr) \to \A_p \bigl( C_{i-1}(H) \bigr) * \A_i$
induces a monomorphism (resp. an epimorphism) in all homology groups of degree~$\leq n-t+i$.
So by (5) (resp.~(4)) of Lemma~\ref{lemmaRelationPhiPsi}, with~$t, n+t-i$ in the roles of~``$j, n$'', 
we see that~$\Phi_{i,t}$ induces monomorphisms (resp. epimorphisms)
in all the homology groups of degree~$\leq (n-t+i) + (t-i)=n$.
So the composition~$\psi_H$ of all the~$\Phi_{i,t}$ induces a monomorphism (resp. an epimorphism) in the~$n$-th homology.
Since the~$n$-th homology group of the domain (resp.~codomain) of~$\psi_H$ is nonzero by hypothesis,
$\psi_H$ does not induce the zero map in the~$n$-th homology group of~$\A_p(H)$---so
we can apply Theorem~\ref{theoremNormalCaseReduction}, and complete the proof.
\end{proof}

To finish this section, we present a local version of Theorem~1 of~\cite{Segev}.
For that purpose, we define a local-diagonal poset $\D_p(H)$,
analogous to that of~\cite{Segev}---which we described after Remark~\ref{rk:initimpl}. 
So let~$L_1,\ldots,L_t$ be an orbit of components of $G$,
and let~$H := \bigcap_i\ N_G(L_i)$ be the local kernel on these components.
Let~$\D_p(H)$ be the subposet of elements~$A \in \A_p(H)$ such that
there exists~$J \subseteq \{ 1,\ldots,t \}$ with $|J| \geq 2$ and~$C_A(\{L_j:j\in J\}) = C_A(L_i)$ for all~$i \in J$.

We now state our version of Theorem~1 of Segev.
In contrast with the original theorem of Segev,
we do not require~$O_{p'}(G) =1$, nor any extra inductive hypothesis on the components:

\begin{theorem}
\label{theorem1}
Let~$L$ be a component of~$G$ of order divisible by~$p$, 
and~$L_1,\ldots,L_t$ its~$G$-orbit.
Let~$H$ be the local kernel~$\bigcap_i\ N_G(L_i)$.
If~$\D_p(H) \to \A_p(H)$ is not surjective in homology, then~$G$ satisfies~(H-QC).
\end{theorem}

\begin{proof}
As usual for~(H-QC) we assume that~$O_p(G)=1$; and we must show that~$\tilde{H}_* \bigl( \A_p(G) \bigr) \neq 0$. 
We first follow the initial arguments in the proof of Segev's Theorem~1 in our discussion in Section~\ref{sectionDiscussionSegev}.
Indeed those arguments also provide the basic setup for the proof of our Theorem~\ref{theoremNormalCaseReduction} in this section.

\smallskip

\noindent
So we begin by essentially quoting a number of features from our proof of~Theorem~\ref{theoremNormalCaseReduction}:

Since we no longer have~$O_{p'}(G) = 1$ and its consequences in Remark~\ref{rk:conseqOp'G=1},
we will use instead the partial replacements that we called~(i)$_L$, (ii)$_L$, and~(iii)$_L$.

We also obtain the usual reduction to~$\A_p(H) \subsetneq \A_p(G)$, and hence nontrivial action~$H < G$: 
Namely in any case for~$H$, the hypothesis of non-surjectivity of a map to~$\tilde{H}_* \bigl( \A_p(H) \bigr)$ 
guarantees that that this homology is at least nonzero;
so that in the trivial-action case where~$\A_p(H) = \A_p(G)$,
this nonzero-ness already includes the conclusion for~(H-QC). 

So using that reduction along with~(iii)$_L$, as in Remark~\ref{rk:conseqH<G}
we obtain~$t \geq p \geq 2$ components in our orbit; 
and nontriviality of the decomposition~$\A_p(G) = Y \cup Z$, with $Y,Z,Y_0,V_0 \neq \emptyset$. 
This last gives the usual Mayer-Vietoris sequence in Remark~\ref{rk:MVseqfordecompYZ}. 
Hence from our implications-analysis in Remark~\ref{rk:initimpl}, it will suffice to prove that~(B) holds.

\bigskip

\noindent
We now follow Segev's argument as in our discussion of the proof of his Theorem~1:

Namely we will show that~$V_0 \subseteq \D_p(H)$.
That is, we need to show that~$b : V_0 \to \A_p(H)$ is not surjective in homology,
where we recall that~$V_0 = \{ E \cap H \tq E \cap H > 1 , E \in \A_p(G) - \A_p(H) \}$.
For if we can establish that non-surjectivity, we can exploit the commutative diagram:
\[\xymatrix{
  \tilde{H}_*(V_0) \ar[d]_b \ar[r]^i & \tilde{H}_*(\D_p(H)) \ar[dl]_d \\
  \tilde{H}_* \bigl( \A_p(H) \bigr)
}\]
where the maps are induced by the inclusions.
Since~$d$ is not surjective by hypothesis, $b = d \circ i$ is not surjective---and hence~(B) holds.
Therefore~$\tilde{H}_*(\A_p(G)) \neq 0$, so that we would then complete the proof.

Thus it remains to prove that~$V_0 \subseteq \D_p(H)$.
The proof follows Segev's original; and indeed essentially the same idea is used for Claim~(C) 
in our proof of Theorem~\ref{theoremNormalCaseReduction}:
We saw above that~$V_0 \neq \emptyset$; so we now consider any~$A \in V_0$:
Then~$A = E \cap H$, where~$E \in \A_p(G) - \A_p(H)$ satisfies~$E \cap H > 1$.
Note that any nontrivial~$e \in E - H$ must induce a nontrivial action on the~$G$-orbit~$\{ L_1 ,\ldots , L_t \}$.
Let~$\O$ be a nontrivial orbit of this~$e$-action---so that~$|\O| = p \geq 2$.
Note that nontrivial elements of~$A \leq C_G(e)$ must exhibit isomorphic actions on all members~$L_i$ of~$\O$; 
hence we get~$C_A(\O) = C_A(L_i)$ for each~$L_i \in \O$, where~$|\O| \geq 2$---and
this implies the condition which defines~$A \in \D_p(H)$.  
So we have now completed the proof.          
\end{proof}

\section{Some more concrete consequences of the generalized method}
\label{sec:concreteconsqH1etc}

In this section, we establish several corollaries of our primary technical Theorem~\ref{theoremNormalCaseReduction};
including our main result Theorem~\ref{theorem2}.

\bigskip

First for Segev's Theorem~2 in~\cite{Segev} (which we stated as Theorem~\ref{theoremSegevOriginal} here),
we will deduce the \textit{local} version from our Theorem~\ref{theoremNormalCaseReduction}---after assuming
one of the inductive hypotheses~(H1) or~(H$L(p)$).
Recall we saw that Theorem~\ref{theoremNormalCaseReduction} is an analogue of Segev's Theorem~2
which assumes the~$\A_p(H)$-form of the hypothesis in Remark~\ref{rk:ApHformhypThm2};
below we instead assume hypotheses on the individual components~$L_i$, 
in analogy with the original hypotheses of Theorem~2:

\begin{corollary}
\label{cor:genTheorem2Segev}
Let~$L$ be a component of~$G$, and~$L_1,\ldots, L_t$ its~$G$-orbit.
We write~$H := \cap_{i}\ N_G(L_i)$ for the local-kernel on the orbit;
and we recall the posets~$\A_i$ from Definition~\ref{defn:CalAi}.
Assume that~$G$ satisfies~(H1) or~(H$L(p)$).

If~$\A_p(L_i) \to \A_i$ is nonzero in homology for each~$i \geq 1$, 
then~$G$ satisfies~(H-QC).

In particular, this holds if~$\A_p(L) \to \A$ is nonzero in homology, where~$\A$ is one of the following:
  \[ \Imageposet{H}{L},\,\, \Imageposet{G}{L},\,\, \A_p \bigl( \Aut_H(L) \bigr),\,\,
                                                   \A_p \bigl( \Aut_G(L) \bigr),\,\, \A_p \bigl( \Aut(L) \bigr) .
  \]
\end{corollary}

\begin{proof}%[Proof of Corollary \ref{coroTheorem2Segev}]
As usual for~(H-QC), we assume that~$O_p(G) = 1$.
By Remark~\ref{remarkH1andHLp}, we can assume that~(H$L(p)$) holds.
In particular, $p$ divides the order of~$L$; 
and as in the proof of~Theorem~\ref{theoremNormalCaseReduction},
we have the partial replacements (i)$_L$, (ii)$_L$, and (iii)$_L$ for Remark~\ref{rk:conseqOp'G=1}. 
And then as in Definition~\ref{defn:CalAi} (which used~(i)$_L$ for faithfulness in Lemma~\ref{lemmaPOuterPreimage}), 
we see that the maps~$\A_p(L_i) \to \A_i$ in the hypothesis may be regarded as inclusions~$\A_p(L_i) \subseteq \A_i$.

\bigskip

\noindent
First, we prove the ``In particular" reduction of the Theorem:

Fix~$i$, and choose~$g \in G$~with $L_i = L^g$.
Let~$\A$ be one of the posets in the statement:
  \[ \Imageposet{H}{L},\,\, \Imageposet{G}{L},\,\, \A_p \bigl( \Aut_H(L) \bigr),\,\,
                                                   \A_p \bigl( \Aut_G(L) \bigr),\,\, \A_p \bigl( \Aut(L) \bigr) .
  \]
Since we saw in Definition~\ref{defn:CalAi} that~$\A_i \subseteq \Aut_H(L_i)$, 
we can extend our earlier inclusion-statement to~$\A_p(L_i) \subseteq \A_i \subseteq \A^g$
for any of these~$\A$---where~$\A^g$ is the analogous~$\A$-poset for~$L^g$ obtained via the conjugation action.
Now we are assuming for the ``In particular'' statement that~$\A_p(L) \to \A$ is not the zero map in homology;
then by the conjugation action, neither is~$\A_p(L_i) = \A_p(L^g) \to \A^g$.
Hence~$\A_p(L_i) \to \A_i$ is not the zero map in homology---and
we have reduced to the main hypothesis of Corollary~\ref{cor:genTheorem2Segev}, as claimed.

\bigskip

\noindent
Now we turn to the main proof of Corollary~\ref{cor:genTheorem2Segev}:

In view of our discussion of the hypothesis, we will see that the proof parallels
the deduction of the~$\A_p(H)$-form of the hypothesis of Segev's Theorem~2,
in our earlier discussion leading up to Remark~\ref{rk:ApHformhypThm2}.

As a brief initial overview:
The earlier direct product group~$E(G) = L_1 \times \cdots \times L_n$
will now be replaced by the ``inner automorphism'' part of~$H$, 
namely~$H_0 := C_H(N) L_1 \cdots L_t$---which this time is instead a {\em central\/} product, over a $p'$-center.  
We still get that~$\A_p(H_0)$ is the join of corresponding factors~$\A_p \bigl( C_H(N) \bigr)$ and ~$\A_p(L_i)$;
with homology given by the tensor product of the homology of the individual terms. 
Now our hypothesis for Corollary~\ref{cor:genTheorem2Segev} is
that each of the latter terms gives a nonzero image in~$\tilde{H}_*(\A_i)$ when~$i \geq 1$;
while we get~$\tilde{H}_*(\ \A_p \bigl( C_H(N) \bigr)\ ) \neq 0$ using~(H$L(p)$), 
and this maps via the identity to ~$\tilde{H}_*(\A_0)$.
Hence the tensor product of these nonzero images is nonzero in the homology~$\tilde{H}_*(X)$ of the join~$X$ of the~$\A_i$.
Thus for~``$\B$'' given by~$\A_p(H_0)$, we get that~$\psi_H |_{\B}$ is nonzero in homology---namely 
the hypothesis of Theorem~\ref{theoremNormalCaseReduction}, which we quote to complete the proof of~(H-QC).

We now expand the above rough outline to a more formal argument: 
Let $\psi_H : \A_p(H) \to X$ denote the map in Theorem~\ref{theoremNormalCaseReduction};
we want to verify the hypothesis there that the map~$\psi_H |_{\B}$ should be nonzero in homology.
Now~$\psi_H |_{\B}$ arises from~$\B = \A_p(H_0) \subseteq \A_p(H)$ followed by~$\psi_H$.
And we still have the analogue of the property in the discussion before Remark~\ref{rk:ApHformhypThm2},
as indicated in Definition~\ref{defn:psiH},
that the restriction of~$\psi_H$ to each~$L_i$ is the ``inclusion''~$\A_p(L_i) \subseteq \A_i$ above 
(and the identity on~$\A_p \bigl( C_H(N) \bigr) = \A_0$).
This means that the restriction~$\psi_{H_0} := \psi_H |_{\A_p(H_0)}$ just gives 
the standard homotopy equivalence of~$\A_p(H_0)$ for the central-product group~$H_0$,
with the join of the $\A_p$-posets for its factors~$C_H(N)$ and~$L_i$
(cf.~\cite[Prop~2.6]{Qui78} or Proposition~\ref{variantQuillenFiber}).
That is, we have the maps:
  \[ \A_p \bigl( H_0 ) \overset{\psi_{H_0}}{\to } \A_p \bigl( C_H(N) \bigr) * \A_p(L_1) \ldots * \A_p(L_t)
                       \overset{j}{\to } \A_0 * \A_1 * \ldots *\A_t = X ,
  \]
where~$\psi_{H_0}$ is a homotopy equivalence, and~$j$ is an inclusion.
In particular, we have~${\psi_H}|_{\B} = j \circ \psi_{H_0}$.
Furthermore, letting~$j_i : \A_p(L_i) \to \A_i$ denote the inclusion,
with~$j_0 : \A_p \bigl( C_H(N) \bigr) \to \A_p \bigl( C_H(N) \bigr)$ the identity map,
we have~$j = j_0 * j_1 * \ldots * j_t$ as a fuller expression for the inclusion in the join of the posets.
Hence the map induced by~$j$ in homology is $j_* = (j_0)_* \otimes (j_1)_*\otimes \ldots \otimes (j_t)_*$.
The map~$j_*$ is nonzero since each map~$(j_i)_*$ for $i \geq 1$ is nonzero by hypothesis,
and we are taking tensor product over a field.
Hence the composition $({\psi_H}|_{\B})_* = (j \circ \psi_{H_0})_* = j_* \circ (\psi_{H_0})_*$
(where~$\psi_{H_0}$ is a homotopy equivalence) is a nonzero map.
This give the hypothesis for Theorem~\ref{theoremNormalCaseReduction}, which we quote to complete the proof of~(H-QC).
\end{proof}

Note that Theorem~\ref{theorem2} is just the particular choice~$\A = \Aut_H(L)$ in Corollary~\ref{cor:genTheorem2Segev}.
The two final statements follow, since they involve the further-specialized case of equality---that is, we have~$L = \Aut_H(L)$,
so that we are considering the image just of the identity map:
and the homology of~$\A_p(L)$ is nonzero by Theorem~\ref{thm:HomologyCalA}, 
since~$L$ is quasisimple with~$p'$-center by~(i)$_L$.

\bigskip

We give next an alternative corollary of Theorem~\ref{theoremNormalCaseReduction}:
It provides an elimination result for the case where (roughly) 
outer $p$-automorphisms do occur, but they are ``separated''---rather than being embedded diagonally across several components:
the model such case would be just a direct product of automorphism groups of individual components.
In particular, the value~$t = 1$ covers the case of a normal component:

\begin{corollary}
\label{cor:theorem2MoreGeneral}
Let~$L$ be a component of~$G$ and~$L_1 , \ldots , L_t$ its~$G$-orbit.
Assume further that~$G$ satisfies~(H1) or~(H$L(p)$), and also that:
\begin{enumerate}[label=(\roman*)]
\item For each~$1 \leq i \leq t$, there exists a subgroup~$F_i$ such that:

\centerline{
   $L_i \leq F_i \leq N_G(L_i)$,\  $[F_i , C_G(L_i)] = 1$,\  and~$F_i \cap C_G(L_i)$ is a~$p'$-group.
            }

\item $\A_p \bigl( N_G(L_i) \bigr) = \A_p \bigl( F_i C_G(L_i) \bigr)$ for all~$i$.
\item If~$i \neq j$ then~$[F_i , F_j] = 1$.
\end{enumerate}
Then~$G$ satisfies~(H-QC).
\end{corollary}

\begin{proof}
In view of~Remark \ref{remarkH1andHLp}, since we are trying to prove~(H-QC), we can work directly under~(H$L(p)$).
As usual for~(H-QC) we assume that~$O_p(G) = 1$.

We will prove that the above special conditions on the~$F_i$ yield~$(\psi_H)_* \neq 0$---giving
the hypotheses of Theorem~\ref{theoremNormalCaseReduction}, and hence~$\tilde{H}_* \bigl( \A_p(G) \bigr ) \neq 0$.

\bigskip

By hypothesis~(i), $F_i \leq C_G \bigl( C_G(L_i) \bigr)$; so~$C_G(L_i) \leq C_G(F_i)$, and hence~$C_G(L_i) = C_G(F_i)$.
It follows that~$Z(F_i) = F_i \cap C_G(F_i)$ is a~$p'$-group, again using~(i).

When~$t \geq 2$ and~$i \neq j$, we get~$[F_i,F_j] = 1$ by hypothesis~(iii), so~$F_i \leq C_G(F_j)$;
and then we see that~$F_i \cap F_j \leq Z(F_i) \cap Z(F_j)$ is a~$p'$-group.
Together with hypothesis~(ii) and Lemma~\ref{lemmaInnerDecomposition}, these observations allow us to conclude that:

\begin{equation}
\label{eq:ApH}
  \A_p(H) = \A_p \bigl( F_1 \ldots F_t C_H(F_1 \ldots F_t) \bigr) , 
\end{equation}

\noindent
and~$C_H( F_1 \ldots F_t ) = C_G( F_1 \ldots F_t ) = C_G( L_1 \ldots L_t )$.
Moreover, $F_1 \ldots F_t C_H( F_1 \ldots F_t) $ is a central product by~$p'$-centers.
Let~$N := L_1 \ldots L_t$, so that~$C_G(N) = C_H(N) = C_H( F_1 \ldots F_t )$.
(Notice that~(\ref{eq:ApH}) and these conditions show that here,
$H$ itself has some of the formal properties of the group~$H_0$ in the proof of Corollary~\ref{cor:genTheorem2Segev}.)

Now we verify the hypotheses of Theorem~\ref{theoremNormalCaseReduction}:
Namely we will show that the map given there
by~$\psi_H : \A_p(H) \to \A_p \bigl( C_H(N) \bigr) * \A_1 * \ldots *\A_t =: X$ is not zero in homology---because 
it is a homotopy equivalence, onto the join~$X$---which has nonzero homology as we will see below.
Recall from Definitions~\ref{defn:CalA} and~\ref{defn:CalAi} 
that we have~$\A_i = \Im(\ \A_p \bigl( C_i(H) \bigr) \to [ \A_p \bigl( \Aut_{C_i(H)}(L_i) \bigr)\ \cup \{ 1 \} ]\ )\ - \{ 1 \}$,
where we recall that~$C_i(H) = C_H( L_{i+1} \ldots L_t )$.

First, we claim that~$\A_i = \A_p \bigl( F_i/Z(F_i) \bigr)$.
When~$t \geq 2$ and~$i \neq j$, by hypothesis~(iii) along with~$C_G(F_j) = C_G(L_j)$ above, we have:
  $$ F_i \leq \bigcap_{j \neq i}\ C_G(F_j) = \bigcap_{j \neq i}\ C_G(L_j) \leq C_i(H) . $$
Moreover using~(i), $[F_i , C_{i-1}(H)] \leq [F_i , C_G(L_i)] = 1$,
and~$F_i \cap C_{i-1}(H) \leq F_i \cap C_G(L_i)$ is a~$p'$-group.
Therefore, if~$A,B \in \A_p(F_i)$ with~$A C_{i-1}(H) = B C_{i-1}(H)$, then it is not hard to see that~$A = B$.
Now for $A \in \A_p \bigl( N_G(L) \bigr)$, by~(ii) we have~$A \leq A_{0} \times A_{1}$,
where~$A_0$ is the projection of~$A$ onto~$F_i$, and~$A_1$ is the projection of~$A$ onto~$C_G(L_i)$.
If we take the quotient by~$C_{i-1}(H)$, we see that~$A C_{i-1}(H) / C_{i-1}(H) = A_0$.
Hence the image of the map that goes from~$\A_p \bigl( C_i(H) \bigr)$ into~$\A_p \bigl( \Aut_{C_i(H)}(L_i) \bigr) \cup \{ 1 \}$
is exactly~$\A_p \bigl( F_i/Z(F_i) \bigr) \cup \{ 1 \}$; so we do indeed have~$\A_i = \A_p \bigl( F_i/Z(F_i) \bigr)$.

Second, we prove that the map $\varphi_i : \A_p\bigl( C_i(H) \bigr) \to \A_p \bigl( C_{i-1}(H) \bigr) * \A_i$
is a homotopy equivalence for each~$i$:
In the case~$A \in \A_p \bigl( F_i/Z(F_i) \bigr)$, $A$ can be viewed in~$\A_p(F_i)$ since~$Z(F_i)$ is a~$p'$-group, so:
  $$ \varphi^{-1}\bigl(\ (\ \A_p \bigl( C_{i-1}(H) \bigr) * \A_i\ )_{\leq A}\ \bigr) = \A_p \bigl( A C_{i-1}(H) \bigr) , $$
which is contractible since~$O_p \bigl( A C_{i-1}(H) \bigr) \geq A > 1$.
In the other case~$A \in \A_p \bigl( C_{i-1}(H) \bigr)$, we see:
  $$ \varphi^{-1} \bigl(\ (\ \A_p \bigl( C_{i-1}(H) \bigr) * A_i\ )_{\leq A}\ \bigr) = \A_p(A) , $$
which is also contractible.
Therefore, by Quillen's fiber lemma (cf. Proposition \ref{variantQuillenFiber}), $\varphi_i$ is a homotopy equivalence.
And then using Lemma \ref{lemmaRelationPhiPsi}:
each~$\Phi_{i,t}$ is a homotopy equivalence via the relation~$\Phi_{i,t} = \varphi_i * \Id_{\A_{i+1}} * \ldots * \Id_{\A_t}$;
and so we get that~$\psi_H$~is a homotopy equivalence,
since we in fact get~$\psi_H = \Phi_{1,t} \circ \ldots \circ \Phi_{t,t}$ as the composition of homotopy equivalences
%We can also see this by noting that $\psi_H$ is the homotopy equivalence of the join of posets,
%after using equation \ref{eqApH} and the fact that $F_1\ldots F_t C_H(N)$ is a central product by $p'$-centers.

It remains to show that~$\psi_H$ is not the zero map in homology.
By~(H$L(p)$), $C_H(N)$ satisfies~(H-QC);
and~$O_p \bigl( C_H(N) \bigr) = 1$, since~$C_H(N)$ is normal in~$G$ with~$O_p(G) = 1$.
Therefore~$\tilde{H}_*(\ \A_p \bigl (C_H(N) \bigr)\ ) \neq 0$.
Now, for each $i$, the quotient~$F_i/Z(F_i)$ is an almost simple group,
so by~\cite{AK90} (cf.~Theorems~\ref{theoremAK} and \ref{thm:HomologyCalA} below)
we have that~$\tilde{H}_*(\ \A_p \bigl( F_i/Z(F_i) \bigr)\ ) \neq 0$.
Finally by the homology decomposition of the join of spaces (see equation~\ref{equationJoinHomology}), we get nonzero homology
for the join~$X = \A_p \bigl( C_H(N) \bigr) * \A_p \bigl( F_1/Z(F_1) ) * \ldots * \A_p \bigl (F_t/Z(F_t) \bigr)$;
and so since~$\psi_H$ is a homotopy equivalence of~$\A_p(H)$ with~$X$, it is not the zero map in homology.
And this is what remained to be established.
\end{proof}

\section{Some properties of the image poset in the generalized method}

This section begins with some properties of the image poset~$\Imageposet{G}{L}$.
We then show~$\tilde{H}_*( \Imageposet{G}{L} ) \neq 0$ when~$L$ is quasisimple---extending
the almost-simple case in Aschbacher-Kleidman~\cite{AK90} to the almost-quasisimple case, 
via posets of this~``$\Aut_G(L)$''-form.
Later, we also give concrete sufficient conditions,
on the centralizers of whatever outer~$p$-automorphisms of a component~$L$ do occur in~$G$, 
which will allow us to establish~(H-QC).

\begin{definition}
\label{defn:admitonlycyclic}
Let~$L$ be a subgroup of~$G$.
If~$A \in \A_p \bigl( N_G(L) \bigr)$ is an elementary abelian~$p$-subgroup such that~$A \cap \bigl( L C_G(L) \bigr) = 1$,
then we say that \textit{$A$ is a~$p$-outer of~$L$ in~$G$}.
(We emphasize that the word outer is here being used
in the ``purely outer'' sense---that~$A$ contains {\em no\/} nontrivial inner automorphisms.)
We define the corresponding posets of~$p$-outers by:
  \[ \Outposet_G(L)       := \{ A \in \A_p \bigl( N_G(L) \bigr) \tq A \text{ is~$p$-outer on }L \} \]
  \[ \hat{\Outposet}_G(L) := \Outposet_G(L) \cup \{ 1 \} . \]
If~$\P \subseteq \A_p(G)$, let~$\Outposet_{\P}(L) := \P \cap \Outposet_G(L)$
and~$\hat{\Outposet}_{\P}(L) := \Outposet_{\P}(L) \cup \{ 1 \}$.

Sometimes we work in subposets~$\A \subseteq \Imageposet{G}{L}$ of the poset in Definition~\ref{defn:CalA}.
Since the members~$A$ of~$\A$ are then quotients of subgroups of~$N_G(L)$ modulo~$C_G(L)$,
the $p$-outer condition becomes in effect just~$A \cap \bigl( L/Z(L) \bigr)  = 1$.
So we get:

\centerline{
     $\Outposet_{\A}(L) = \A - \N_{\A}(L)$ ,
            }

\noindent
in the language of the inflated poset~$\N(-)$.

We say that~\textit{$L$ admits only cyclic~$p$-outers in~$G$}
if~$\Outposet_G(L)$ is non-empty, and its members are cyclic (rather than of rank~$\geq 2$).
This means that the~$p$-outers have order~$p$---and there are no poset inclusions among them. 
For example, if~$\Out(L)$ has cyclic Sylow~$p$-subgroups, then either the cyclic-only case or the empty case holds.

Later we will use the following properties: 
Assume that~$A \in \Outposet_G(L)$, for an~$L$ which admits only cyclic~$p$-outers; set~$\A :=\A_{G,L}$
and~$\N := \N_{\A}(L)$.
Then when~$E \in \A_{> A}$, by the order-$p$ property of cyclic-only,
we have~$E = A \times B$ for nontrivial~$B \in \A_p(L)$, so that $E \in \N_{\A}(L)$;
thus we see that~$\A_{> A} = \N_{> A}$.
\donerk
\end{definition}

We now use the above notion of~$p$-outers
to further study the properties of the image poset~$\Imageposet{G}{L}$ in Definition~\ref{defn:CalA}.
So recall that if~$L \leq G$ is a subgroup
and $\pi : \A_p \bigl( N_G(L) \bigr) \to \A_p \bigl( \Aut_G(L) \bigr) \cup \{ 1 \}$
is the map induced by the quotient $N_G(L) \to \Aut_G(L)$, then:
  \[ \Imageposet{G}{L} = \pi(\ \A_p \bigl( N_G(L) \bigr)\ ) - \{ 1 \} = \pi(\ \A_p \bigl( N_G(L) \bigr) - \A_p \bigl( C_G(L) \bigr)\ ). \]
In the following lemma, we will characterize the poset~$\Imageposet{G}{L}$
in terms of the~$p$-outers~${\Outposet}_G(L)$ of~$L$.
And then after that,
we will prove that~$\tilde{H}_*(\Imageposet{G}{L}) \neq 0$ if~$L$ is a quasisimple subgroup of~$G$, by using~\cite{AK90}.
This quasisimple result can be seen as a small improvement of the almost-simple case of Quillen's conjecture in~\cite{AK90},
since we show that the conjecture also holds for the poset~$\Imageposet{G}{L}$ consisting of quotients,
and not only for posets consisting of subgroups, such as~$\A_p(T)$ with~$F^*(T) = L$.

Thus as just indicated, we first provide some different descriptions of the image poset.
(We use the bar notation~$\bar{E}$ to denote the image of an element~$E$ under the map~$\pi$---notably when $E$ is a~$p$-outer.)

\begin{lemma}[{Description of~$\Imageposet{G}{L}$}]
\label{lemmaPOuterPreimage}
Let~$L \leq G$ be a subgroup with~$p'$-center, and let:
  $$ \pi : \A_p \bigl( N_G(L) \bigr) \to \A_p \bigl( \Aut_G(L) \bigr) \cup  \{ 1 \} $$
be the map induced by the quotient $N_G(L) \to \Aut_G(L)$.
Then since $Z(L)$ is a $p'$-group, $\pi$ embeds~$\A_p(L) \cong \A_p(\overline{L})$ into~$\A_p \bigl( \Aut_G(L) \bigr)$,
and hence into~$\A_{G,L}$.
So we will identify~$\A_p(L)$ with its (barred) image via~$\pi$;
and it is also convenient if slightly abusive to just write unbarred~$L$ in place of~$\overline{L}$ itself.
Then we have that:
  \[ \Imageposet{G}{L} = \bigcup_{E \in \hat{\Outposet}_G(L)}\ \A_p(L \overline{E}) . \]
Equivalently, $\Imageposet{G}{L}$ is isomorphic to the poset:
  \[     \{\ EC_G(L) \tq E \in \A_p \bigl( N_G(L) \bigr) - \N_G \bigl( C_G(L) \bigr)\ \}
       = \{\ EC_G(L) \tq E \in \A_p \bigl( N_G(L) \bigr) ,\ C_E(L) = 1\ \} .
  \]
In particular, we have that:
  $$ \Outposet_{\Imageposet{G}{L}}(L) = \{ \overline{E} \tq E \in \Outposet_G(L) \} , $$
and if~$\Outposet_G(L) = \emptyset$ (i.e.~there are no~$p$-outers), then~$\Imageposet{G}{L} = \A_p(L)$.
\end{lemma}

\begin{proof}
Since~$Z(L)$ is a~$p'$-group by hypothesis, we see that:

  \quad \quad \quad Members of~$\A_p(L)$ act faithfully on~$L$.

\noindent
In particular, 
we have a natural poset isomorphism~$\A_p(L) \to \A_p \bigl( L/Z(L) \bigr)$ induced by the quotient map~$L \to L/Z(L)$.
Therefore, without loss of generality:

  \quad \quad \quad {\em In the remainder of the proof, we assume that~$Z(L) = 1$\/}.

\noindent
Let~$A,B \in \A_p(L)$ such that~$\bar{A} \leq \bar{B}$.
Then~$AC_G(L) \leq BC_G(L)$.
Since~$A$ and~$B$ commute with~$C_G(L)$, and $ A\cap C_G(L) = 1 = B \cap C_G(L)$ by the condition~$Z(L) = 1$, we get~$A \leq B$.
This shows that~$\A_p(L)$ embeds into~$\A_p \bigl( \Aut_G(L) \bigr)$ naturally via~$\pi$.

Now we show that~$\Imageposet{G}{L} = \bigcup_{E\in \hat{\Outposet}_G(L)}\ \A_p(L \overline{E})$:
Consider any~$E \in \A_p \bigl( N_G(L) \bigr)$.
Then we can write~$E = \bigl( E \cap L C_G(L) \bigr) \times E_1$,
for some complement~$E_1$ to~$E \cap  \bigl( LC_G(L) \bigr)$.
So for that purely-outer part we have~$E_1 \in \hat{\Outposet}_G(L)$.
Write~$E_0$ for the projection of~$E \cap \bigl( LC_G(L) \bigr)$ into~$L$.
Note that~$E_0$ and~$E_1$ commute, and $(E_0 E_1) \cap C_G(L) = 1$.
Then:
  \[ \overline{E} = E C_G(L)\ /\ C_G(L) = E_0 E_1 C_G(L)\  /\ C_G(L) \groupiso E_0 E_1 .\]
Finally, since~$E_0 \leq L$, we have that~$\overline{E} \in \A_p(L \overline{E_1})$.

By the isomorphism theorems,
the poset~$\Imageposet{G}{L}$ is isomorphic to the poset of subgroups:
  $$ \{ EC_G(L) \tq E \in \A_p \bigl( N_G(L) \bigr) - \A_p \bigl( C_G(L) \bigr) \} . $$
If~$E \in \A_p \bigl( N_G(L) \bigr) - \A_p \bigl( C_G(L) \bigr)$, write~$E = E_0 C_E(L)$ for some complement~$E_0$ to~$C_E(L)$ in~$E$.
In this case note that~$E_0 \neq 1$.
Then~$E C_G(L) = E_0 C_G(L)$, and~$E_0 \notin \N_G \bigl( C_G(L) \bigr)$.
Therefore:
\begin{align*}
  \Imageposet{G}{L} &\cong \{E C_G(L) \tq E \in \A_p \bigl( N_G(L) \bigr) - \A_p \bigl( C_G(L) \bigr) \} \\
                    &=     \{ E C_G(L) \tq E \in \A_p \bigl( N_G(L) \bigr) - \N_G \bigl( C_G(L) \bigr) \} .
\end{align*}
The ``In particular" part is clear from these descriptions of~$\Imageposet{G}{L}$.
\end{proof}

In particular if~$L$ is simple, then it acts faithfully on itself.
And in that case, we can regard its mapping into the quotient~$\Aut_G(L)$ just as an inclusion;
and so using Lemma~\ref{lemmaPOuterPreimage}
we can write~$\Imageposet{G}{L} = \bigcup_{A \in \hat{\Outposet}_G(L)}\ \A_p(L\bar{A}) \subseteq \A_p \bigl( \Aut_G(L) \bigr)$.
This subposet may not be a poset of subgroups;
i.e.~may not be of the form~$\A_p(T)$ for some almost-simple group~$T \leq \Aut_G(L)$---but it is somewhat analogous.
Our aim now is to show that~$\Imageposet{G}{L}$ behaves like the Quillen poset of an almost-simple group,
so that~$\tilde{H}_*(\Imageposet{G}{L}) \neq 0$.
We use the generalized version of Robinson's Lemma due to Aschbacher-Smith~\cite[Sec~5]{AS93},
together with the proofs in~\cite{AK90} of the almost-simple case of the conjecture.

\begin{lemma}[{\cite[Lemma~0.14]{AS93}}]
\label{lemmaRobinson}
Suppose that a~$q$-hyperelementary%
\footnote{
   Recall that a~$q$-hyperelementary group is a group~$H$ such that~$O^q(H)$ is cyclic.
          }
group~$H$ acts on a poset~$Y$ with~$\tilde{H}_*(Y) = 0$.
Then~$\tilde{\chi}(Y^H) \equiv 0 \mod q$.
In particular, $Y^H$ is non-empty.
\end{lemma}

The almost simple case of~(H-QC) is a consequence of the above lemma and the main theorems of~\cite{AK90}.
We summarize the results that we will need from \cite{AK90} in the theorem below.

\begin{theorem}[{Aschbacher-Kleidman}]
\label{theoremAK}
Let~$T$ be almost-simple. The following are equivalent:
\begin{enumerate}
\item For all hyperelementary nilpotent~$p'$-subgroups~$H$ of~$F^*(T)$, the fixed-point set~$\S_p(T)^H$ is not empty. 
\item $p = 2$, $F^*(T) = L_3(4)$ and~$4 \mid |T:F^*(T)|$.
\end{enumerate}
Moreover, (H-QC) holds for almost-simple groups.
\end{theorem}

Theorem~3 of~\cite{AK90} establishes~(H-QC) for almost-simple groups by using the equivalence of the above theorem.
However, the proof of Theorem~3 has a small gap---which can be easily fixed, as we show below:

\begin{theorem}
[{Homology for $\Imageposet{G}{L}$}]
\label{thm:HomologyCalA}
Suppose that~$L \leq G$ is a quasisimple subgroup with~$p'$-center, and let~$\Imageposet{G}{L}$ be the image poset.
Then~$\tilde{H}_*(\Imageposet{G}{L}) \neq 0$.
\end{theorem}

\begin{proof}
Without loss of generality, we can assume that~$Z(L) = 1$.
Also write~$\A := \Imageposet{G}{L}$.
Recall from Lemma~\ref{lemmaPOuterPreimage} that~$\A_p(L)\subseteq \A$.

Assume first that part~(2) of Theorem~\ref{theoremAK} does not hold for $T := \gen{\A} \leq \Aut_G(L)$.
Then there exists a $q$-hyperelementary $p'$-subgroup $H\leq L$ such that $\S_p(T)^H$ is the empty set.
Note that $H$ acts on $\A$, and in particular, $\A^H \subseteq \S_p(T)^H$ is empty.
By Lemma \ref{lemmaRobinson}, $\tilde{H}_*(\A) \neq 0$.

Thus we may assume instead that part~(2) of Theorem~\ref{theoremAK} does hold; that is, $L = L_3(4)$ and~$4 \mid |T:L|$.
In this case, following the original proof of the almost-simple case on p211 of~\cite{AK90},
we take~$P \in \Syl_5(L)$, which is cyclic of order $5$.
In the original proof of Aschbacher-Kleidman, it is stated that~$\A_p \bigl( \Aut(L) \bigr)^P$ has exactly three points:
a subgroup of order~$2$, and two subgroups of order~$16$.
But this is not correct, since we have in fact five points in~$\A_p(\Aut(L))^P$.
Namely, $\A_p(L)^P$ is discrete with two points (two subgroups of order~$16$),
and $\A_p \bigl( \Aut(L) \bigr)^P$ is discrete with five points (two subgroups of order~$16$, and three subgroups of order~$2$).
In any case, if a subposet~$Y \subseteq \A_p \bigl( \Aut(L) \bigr)$ containing~$\A_p(L)$ is stable under the action of~$P$,
then~$Y^P$ is a discrete poset of cardinality between~$2$ and~$5$.
Therefore, $\tilde{\chi}(Y^P)$ is between $1$ and $4$ mod $5$, and in particular is nonzero.
By Lemma~\ref{lemmaRobinson}, $\tilde{H}_*(Y) \neq 0$.
In particular, this holds if~$Y = \A$, since $P \leq L$ and $L$ acts on $\A$;
and also holds if~$Y = \A_p(T)$ for some almost-simple group~$T$ with~$F^*(T) = L$.
This fixes the small gap in~\cite{AK90}, for showing nonzero homology of~$\A_p(T)$; and also finishes our proof for~$\A$.
\end{proof}

\begin{remark}
\label{remarkStrongHomologyCalA}
The above theorem can be strengthened (details will appear elsewhere)
to show that~$\Imageposet{G}{L}$ in fact satisfies the Lefschetz-character version of Quillen's conjecture.
That is, we have the stronger conclusion that~$\tilde{\chi}(\Imageposet{G}{L}^g) \neq 0$ for some~$g \in L$.
\donerk
\end{remark}

\begin{remark}
Lemma~\ref{lemmaPOuterPreimage} and Theorem~\ref{thm:HomologyCalA} provide some properties for the image poset~$\Imageposet{G}{L}$,
for~$L \leq G$ quasisimple with~$p'$-center, which suggest that it behaves much like an~$\A_p$-poset.
Moreover, in view of Lemma~\ref{lemmaPOuterPreimage},
the image poset~$\Imageposet{G}{L}$ can be easily described
in terms of the~$p$-outer poset~$\Outposet_G(L) = \A_p \bigl( N_G(L) \bigr) - \N_G \bigl( L C_G(L) \bigr)$.
On the other hand, the potentially-larger poset~$\A_p \bigl( \Aut_G(L) \bigr)$
cannot necessarily be described in terms of the elements of~$\A_p \bigl( N_G(L) \bigr)$:
since in general the map $\pi : \A_p \bigl( N_G(L) \bigr) \to \A_p \bigl( \Aut_G(L) \bigr) \cup \{ 1 \}$ is not surjective.
That is, in some situations, taking quotient by~$C_G(L)$ might be generating new~$p$-outers in~$L$,
which were not contained in~$N_G(L)$ before---as the example below shows.

The above remarks provide motivation for using the image poset~$\Imageposet{G}{L}$:
indeed we have given above a description of its elements---in terms of the~$p$-outers that are already contained in~$G$.
Moreover, Theorem~\ref{theoremNormalCaseReduction} shows that in fact we can reduce our~(H-QC) analysis
to studying this poset, and the behavior of~$\A_p(L) \to \Imageposet{G}{L}$ in homology.
Even more, in some situations it can be shown that there is a topological section,
and hence an inclusion of~$\Imageposet{G}{L}$ into $\A_p \bigl( N_G(L) \bigr)$---though we won't require such a section in our arguments.
\donerk
\end{remark}

\begin{example}
Let~$L$ be a simple group of Lie type over the field of~$q^{p^2}$ elements.
Let~$A \leq \Aut(L)$ be a cyclic group of field automorphisms of order~$p^2$ of the field $q^{p^2}$, and let~$a$ generate~$A$.
We use~$a$ to define different actions on two distinct components isomorphic to~$L$;
take two copies~$L_1$ and~$L_2$ of~$L$, and let~$B := \langle b \rangle$ be defined via the morphisms:
  \[\phi_1 : b \mapsto a^p \in \Aut(L_1) , \]
  \[\phi_2 : b \mapsto a \in \Aut(L_2) . \]
Consider the semidirect product~$G := (L_1 \times L_2) \rtimes_{\phi_1 \times \phi_2} B$.
We have that~$C_G(L_1) = L_2 \langle b^p \rangle$.
Thus cosets which have order~$p$ in the quotient~$B/L_1$ have elements inducing only inner automorphisms of~$L_1$;
so that elements of~$B$ which induce a nontrivial outer automorphism on~$L_1$ 
must lie in cosets which have order~$p^2$ in the quotient,
and so those elements have order~$p^2$ (or more) rather than~$p$, as elements of~$G$.
It follows then that~$\Outposet_G(L_1) = \emptyset$; and hence~$\Imageposet{G}{L_1} = \A_p(L_1)$.
On the other hand, we have~$\Aut_G(L_1) = G/C_G(L_1) \groupiso L_1 (B / \langle b^p \rangle) \groupiso L_1 \langle a^p\rangle$;
so we see that ~$\Imageposet{G}{L_1} = \A_p(L_1) \subsetneq \A_p \bigl( \Aut_G(L_1) \bigr)$.

Thus the~$p$-outers of~$L_1$ in~$G$ do not suffice to exhibit~$\A_p \bigl( \Aut_G(L_1) \bigr)$.
Moreover, $\Imageposet{G}{L_1}=\A_p(L_1)$ and $\A_p(\Aut_G(L_1))=\A_p(L_1\langle a^p\rangle)$ are not homotopy equivalent in general.
\donerk
\end{example}

We close this section with a useful proposition that gives a sufficient condition---in terms of 
the centralizers of whatever~$p$-outers of~$L$ do arise in~$G$, in its hypothesis (2)---to show
that the map~$\A_p(L) \to \Imageposet{G}{L}$ is nonzero in homology.
Thus it establishes~(H-QC) for~$G$ via Corollary~\ref{cor:genTheorem2Segev}, with~$\Imageposet{G}{L}$ in the role of~``$\A$'':

\begin{proposition}
\label{prop:TrivialInclusionCentralizer}
Suppose that~$G$ satisfies~(H1).
Let~$L$ be a component of~$G$ satisfying:
\begin{enumerate}
\item $\Outposet_G(L)$ is either empty, or it consists only of cyclic~$p$-outers. \\
      Furthermore, there exists~$k \geq 0$ such that:
\item the induced map~$\tilde{H}_k(\ \A_p \bigl( C_L(E) \bigr)\ ) \to \tilde{H}_k \bigl( \A_p(L) \bigr)$
      is the zero map for all~$E \in \Outposet_G(L)$; and
\item $\tilde{H}_k \bigl( \A_p(L) \bigr) \neq 0$.
\end{enumerate}
Then~$G$ satisfies~(H-QC).
\end{proposition}

\begin{proof}
As usual for~(H-QC) we may assume that~$O_p(G) = 1$.
By Theorem~\ref{generalReduction}, we can also assume that~$O_{p'}(G) = 1$.
Set~$\A := \Imageposet{G}{L}$ and~$\N := \N_{\A}(L)$;
We will apply Corollary~\ref{cor:genTheorem2Segev},
by establishing its hypothesis that~$\A_p(L) \to \A$ is not the zero map in homology.

As in Lemma~\ref{lemmaPOuterPreimage} we have:
  $$ \A = \bigcup_{E \in \hat{\Outposet}_G(L)}\ \A_p(L \overline{E}) \subseteq \A_p \bigl( \Aut_G(L) \bigr) . $$
Recall also that $\Outposet_\A(L) = \{ \overline{E} \tq E \in \Outposet_G(L) \}$.
Applying hypothesis~(1) to such~$E$, we get the analogue for the resulting~$\overline{E}$:
namely we see that~$\Outposet_{\A}(L)$ is either empty, or it consists only of cyclic~$p$-outers.
Now since the remainder of our argument takes place entirely in~$\A$,

\centerline{
  {\em below we simplify notation by writing just~$E$ (rather than~$\overline{E}$) for members of~$\A$\/}. 
            }

\noindent
Recall from Definition~\ref{defn:admitonlycyclic}
that since~$\A \subseteq \Aut(L)$, we have~~$\A - \N = \Outposet_{\A}(L)$.

\bigskip

Assume first that~$\Outposet_{\A}(L)$ is empty.
Then~$\A = \N = \A_p(L)$, so that~$\A_p(L) \to \A$ is the identity map; 
and the homology of~$\A_p(L)$ is nonzero using hypothesis~(3);
so by Corollary~\ref{cor:genTheorem2Segev} we get our conclusion of~(H-QC).
(Recall this special case $\A_p(L) = \A$ gives the final statement in Theorem~\ref{theorem2}.)

\bigskip

So we now assume instead that~$\Outposet_{\A}(L)$ is nonempty; and we consider some member~$E$.
Recall (as in Remark~\ref{remarkPropertiesOfCalN})
that we have the homotopy equivalence~$\N_{> E} \simeq \A_p \bigl( C_L(E) \bigr)$,
via the usual retraction~$r : A \mapsto A \cap L$, with homotopy inverse~$A \mapsto AE$;
with a similar equivalence more generally for~$\N \simeq \A_p(L)$. 
We will informally use~$i$ to denote these homotopy inverses.

Let~$s$ denote the number of choices for~$E$; that is, $\Outposet_{\A}(L) = \{ E_1 ,\ldots , E_s \}$. 
Since these give the members of~$\A - \N$, 
we can proceed by building up from~$\N$ to~$\A$---via ``adding one~$E_i$ at a time''. 
Actually it will be more convenient, for use in our Mayer-Vietoris sequence below,
to in fact adjoin each time the larger set~$\A_{\geq E_i}$: 
though in fact only the bottom member~$E_i$ is really being {\em newly\/} added at that point---since 
in view of our cyclic-only hypothesis~(1), from Definition~\ref{defn:admitonlycyclic}
we have~$\A_{> E_i} \subseteq \N$.
Namely we set~$X_0 := \N$, and define the general term inductively via:

\centerline{
    $X_{i+1} := X_i \cup \A_{ \geq E_{i+1} }$ ;
            }

\noindent
in particular, we get~$X_s = \A$.  
Furthermore using Definition~\ref{defn:admitonlycyclic} as just indicated, we get:

\centerline{
    $\A_{ \geq E_{i+1} } \cap X_i = ( \{ E_{i+1}\} \cup \A_{> E_{i+1} } ) \cap X_i = \N_{> E_{i+1} } \cap X_i
                                                                                   = \N_{> E_{i+1} } $ .
            }

\noindent
Now consider the Mayer-Vietoris exact sequence corresponding to the above decomposition of~$X_{i+1}$. 
Since the added-in posets $\A_{\geq E_i}$ are contractible, they give zero-terms in reduced homology;
hence the sequence in degree~$k$ takes the form:
  \[ \xymatrix{
        \ldots \ar[r] & \tilde{H}_{k+1}(X_{i+1}) \ar[r] & \tilde{H}_k( \N_{ > E_{i+1} } ) \ar[r] & \tilde{H}_k(X_i) \ar[r]
                                            	        & \tilde{H}_{k}(X_{i+1}) \ar[r]         & \ldots \\
                      & & \tilde{H}_k(\ \A_p(C_L(E_{i+1}) )\ ) \ar[u]^{\groupiso}_{i_*} \ar[r]_{\phantom{0}\ \; \; =0}
					 & \tilde{H}_k \bigl( \A_p(L) \bigr) \ar[u]
               }
  \]
where the map in the lower line is zero in view of hypothesis~(2).
That determines by composition a zero map into~$\tilde{H}_k(X_i)$---which by commutativity
gives zero for the other composite map into~$\tilde{H}_k(X_i)$.
In view of the isomorphism~$i_*$ induced by the homotopy inverse~$i$ mentioned earlier,
we conclude that~$\tilde{H}_k( \N_{> E_{i+1} } ) \to \tilde{H}_k(X_i)$ is also zero.
Exactness in the sequence then gives a monomorphism:
  \[ \tilde{H}_k(X_i) \hookrightarrow \tilde{H}_{k}(X_{i+1}) . \]
Composing these maps over all~$i$ then gives a monomorphism:
  \[ \tilde{H}_k \bigl( \A_p(L) \bigr) \overset{\groupiso}{\to} \tilde{H}_k \bigl( \N_\A(L) \bigr)
							   = \tilde{H}_k(\N) = \tilde{H}_k(X_0)
                                                           \hookrightarrow \tilde{H}_k(X_s) = \tilde{H}_k(\A) ,
  \]
where the isomorphism is induced by the more general homotopy inverse~$i$ mentioned above.
Note that this map is induced by the inclusion~$\A_p(L) \to \A$,
and it is nonzero if~$\tilde{H}_k \bigl( \A_p(L) \bigr) \neq 0$---which holds here by hypothesis~(3).
By Corollary~\ref{cor:genTheorem2Segev}, we conclude that~$\tilde{H}_* \bigl( \A_p(G) \bigr) \neq 0$---as required for~(H-QC).
\end{proof}

\begin{remark}
\label{rk:suffget(2)forakin(1)}
Hypothesis~(1) of Proposition~\ref{prop:TrivialInclusionCentralizer} may hold reasonably often:
for example, when the~$p$-Sylows of~$\Out(L)$ are cyclic, as we observed in Definition~\ref{defn:admitonlycyclic}.
And notice that hypothesis~(3) {\em always\/} holds---for at least some value of~$k$, in view of Theorem~\ref{thm:HomologyCalA}.
So it would then suffice to verify hypothesis~(2)
just for any of the values of~$k$ that already do satisfy hypothesis~(3).
\donerk
\end{remark}

\section{Elimination of components of sporadic type~\texorpdfstring{$\HS$}{HS} (for the prime 2)}

In this section, we prove part~(2) of Corollary~\ref{corollaryComponents}:
That is, we show that under~(H1), if~$G$ has a component isomorphic to~$\HS$ (the Higman-Sims sporadic group),
then~$G$ satisfies (H-QC).
Recall that~$\Out(\HS)$ is of order~$2$:
thus for odd~$p$, we have~$\A_p(\HS) = \A_p \bigl( \Aut(\HS) \bigr)$,
so that we can immediately apply Corollary~\ref{cor:genTheorem2Segev} to get~(H-QC);
hence in this section, we assume instead that~$p = 2$. 

Here we will still proceed via Corollary~\ref{cor:genTheorem2Segev}:
we will prove that~$\A_2(\HS) \to \A_2 \bigl( \Aut(\HS) \bigr)$ is not the zero map in homology.
To that end, we will inspect the second homology groups of these posets---via examining their Euler characteristics,
and using the structure of the centralizers of the~$2$-outers.
We performed some of the relevant computations in~GAP~\cite{GAP} with the package~\cite{Posets}.

Thus the main result of this section is the following theorem:

\begin{theorem}
\label{theoremHS}
Take~$p := 2$.
Let~$L := \HS$ denote the Higman-Sims sporadic group, and also take~$A := \Aut(\HS)$.
Then we have~$\tilde{\chi} \bigl( \A_2(L) \bigr) = 1767424$ and~$\tilde{\chi} \bigl( \A_2(A) \bigr) = 1204224$;
and the following hold:
\begin{enumerate}
\item $m_2(L) = 4$ and~$m_2(A) = 5$.
\item $\A_2(L) \to \A_2(A)$ is a~$2$-equivalence.
\item $\tilde{H}_n \bigl( \A_2(A) \bigr) = 0$ for~$n\geq 4$.
\item $\tilde{H}_3 \bigl( \A_2(L) \bigr) \subseteq \tilde{H}_3 \bigl( \A_2(A) \bigr)$.
\end{enumerate}
In particular, if~$G$ satisfies~(H1) and~$L$ is a component of~$G$, then~$G$ satisfies~(H-QC) for~$p = 2$.
\end{theorem}

For the ``In particular" part of the above theorem:
Note that~$\Out(HS)$ is cyclic of order~$2$, so we have the cyclic-only situation---which
would give us hypothesis~(1) of Proposition~\ref{prop:TrivialInclusionCentralizer};
however, we do not get the zero-maps for centralizers which would be needed for hypothesis~(2) there.
Instead, we will show that we can still exploit a Mayer-Vietoris sequence
very similar to the one used in proving that Proposition---but now based just on the subset of un-cooperative centralizers; 
in order to then show that~$\A_2(\HS) \to \A_2 \bigl( \Aut(\HS) \bigr)$ is nonzero in homology,
and so again complete the proof of~(H-QC) as there, via Corollary~\ref{cor:genTheorem2Segev}.

Recall from Definition~\ref{definitionCalN} that if~$H \leq G$, then~$\N_G(H) = \{ B \in \A_p(G) \tq B \cap H \neq 1 \}$.
By Remark~\ref{remarkPropertiesOfCalN},
this poset is homotopy equivalent to~$\A_p(H)$ via the retraction~$r : \N_G(H) \to \A_p(H)$ defined by~$B \mapsto B \cap H$,
with inverse given by the inclusion~$\A_p(H) \hookrightarrow \N_G(H)$.
Moreover, we saw there that
if we have~$E \in \A_p(G) - \N_G(H)$, then~$\N_G(H)_{>E}$ is homotopy equivalent to~$\A_p \bigl( C_H(E) \bigr)$
via the same retraction~$r(B) = B \cap H$, but now with the inverse~$i : \A_p \bigl( C_H(E) \bigr) \to \N_G(H)$ given by~$i(B)= BE$.

Also recall the formula:
  \[ \tilde{\chi} \bigl( \A_p(G) \bigr) = \sum_{E \in \A_p(G) \cup \{ 1 \}}\ (-1)^{m_p(E)-1}p^{m_p(E)(m_p(E)-1)/2}\ ; \]
see for example~\cite{JM}.

\begin{proof}[Proof of Theorem~\ref{theoremHS}]
The values of the Euler characteristic in the statement follow from the above formula via direct computation---e.g.~in~GAP.

\bigskip

In the following argument for conclusions~(1) and~(2), we refer to Table~5.3m of~\cite{GLS98}
for the assertions on the structure of the subgroups of~$\Aut(\HS)$.
Since~$L$ has index~$2$ in~$A$, $\Outposet_A(L)$ consists only of cyclic~$2$-outers.
Indeed, if~$E\in \Outposet_A(L)$, then~$E$ is of type~$2C$ or~$2D$, in view of Table~5.3m.
We recall below the structure of the centralizers of these involutions:

The centralizer~$C_L(2C)$ is a non-split extension: given by a normal elementary abelian~$2$-group of~$2$-rank~$4$,
under action by~$O_4^-(2)$.
In particular, $O_2 \bigl( C_L(2C) \bigl) > 1$ and so~$\A_2 \bigl( C_L(2C) \bigr)$ is contractible.
So the inclusion of this centralizer in~$L$ is certainly zero in homology.
Conclusion~(1) follows from Table~5.6.1 of~\cite{GLS98} and the fact that~$m_2 \bigl( C_L(2C) 2C \bigr) = 5$.
(Indeed using the centralizer in the following paragraph, we also get~$m_2 \bigl( C_L(2D) 2D \bigr) = 5$.)

On the other hand, $C_L(2D) \groupiso \SS_8$, and~$\A_2(\SS_8)$ is homotopy equivalent to a wedge of~$512$ spheres of dimension~$2$:
This holds since~$\A_2(\SS_8)$ is simply connected,
while the Bouc poset~$\B_2(\SS_8)$ of non-trivial radical~$2$-subgroups has dimension~$2$,
and is homotopy equivalent to~$\A_2(\SS_8)$---where~$\tilde{\chi} \bigl( \A_2(\SS_8) \bigr) = 512$.
These assertions can be directly computed, or else checked with GAP and the package \cite{Posets}.
In particular, it is not clear that the inclusion of this centralizer in~$L$
should induce the zero map in homology---notably in dimension~$2$.
So the~$2D$-centralizers do not allow us to complete hypothesis~(2) for Proposition~\ref{prop:TrivialInclusionCentralizer}.

Write~$\A := \A_2(A)$, and~$\N := \N_{\A}(L)$.
Since we are working in~$A = \Aut(L)$, using Definition ~\ref{defn:admitonlycyclic} we have~$\A - \N = I_{\A}(L)$;
indeed we mentioned earlier that we have the cyclic-only there, 
so that~$I_{\A}(L)$ consists of groups of order~$p$, with no inclusion relations among them. 

Moreover, if~$E \in I_{\A}(L)$ and~$E$ is of type~$2C$,
then~$\N_{>E} \overset{r}{\simeq} \A_2 \bigl( C_L(2C) \bigr)$ (using the earlier retraction~$r$) is contractible.
Correspondingly we set~$J{_A}(L) := \{ E \in \A \tq E \text{ is of type } 2D \}$;
and~$\N_J := \N \cup \{ \A_{\geq E} : E \in J_{\A}(L) \} $.
Then we see that~$\N_J \hookrightarrow \A$ is a homotopy equivalence,
in view of Proposition~\ref{variantQuillenFiber}---since beyond~$\N_J$,
we are adding only~$\A_{\geq E}$ for~$E$ of type~$2C$---and this poset is~$E * \A_{> E}$, 
where~$\A_{> E}$ is contractible.
Further, the inclusion~$\N \hookrightarrow \N_J$ is a~$2$-equivalence---also via Proposition~\ref{variantQuillenFiber}:
for if~$E \in J_{\A}(L)$, then again using Definition~\ref{defn:admitonlycyclic}, and the earlier retraction~$r$:
  \[ \A_{> E} = \N_{>E} \overset{r}{\simeq} \A_2 \bigl( C_L(E) \bigr) \cong \A_2 \bigl( C_L(2D) \bigr) , \]
which we saw is a wedge of~$2$-spheres, and so is~$1$-connected.
Since~$\A_2(L) \hookrightarrow \N$ is a homotopy equivalence (cf.~Remark~\ref{remarkPropertiesOfCalN}),
we conclude that the composition~$\A_2(L) \subseteq \N \subseteq \N_J \subseteq \A = \A_2(A)$ of our three inclusions,
giving~$\A_2(L) \hookrightarrow \A_2(A)$, is a~$2$-equivalence---and hence we have established conclusion~(2).

\bigskip

For conclusions~(3) and~(4):
We saw that~$\N_J \simeq \A = \A_2(A)$;
and we obtained~$\N_J$ from~$\N$ by a process of adding the~$\A_{\geq E}$ for~$E$ of type~$2D$,
glued through the~$2$-spheres in the link~$\Lk_{\A}(E)$.
Therefore, the process does not change the homology groups of degree~$\geq 4$;
so we conclude that~$\tilde{H}_n \bigl( \A_2(A) \bigr) = \tilde{H}_n \bigl( \A_2(L) \bigr) = 0$ for all~$n \geq 4$.
We now examine that process in more detail---paralleling
the proof of Proposition~\ref{prop:TrivialInclusionCentralizer}:
Assume we have~$s$ members of type~$2D$; order them as~$J_{\A}(L) = \{ E_1 , \ldots , E_s \}$.
Again set~$X_0 := \N$, and:
\begin{equation}
\label{eqDecompositionXi+1}
   X_{i+1} := X_i \cup \A_{\geq E_{i+1}} ;
\end{equation}
in particular, we now get~$X_s = \N_J$.
Again we have:
  \[ \A_{\geq E_{i+1}} \cap X_i = \N_{>E_{i+1}}
                                \overset{r}{\simeq} \A_2 \bigl( C_L(E_{i+1}) \bigr ) \cong \A_2 \bigl( C_L(2D) \bigr) , \]
which we saw is a wedge of~$2$-spheres.
We apply the Mayer-Vietoris sequence to the decomposition of~$X_{i+1}$
given in the right side of equation~(\ref{eqDecompositionXi+1}).
Below we describe the relevant terms of this long exact sequence. 
Note since each $\N_{>E}$ is a wedge of~$2$-spheres that we get values of~$0$
in the corresponding places for its homology in dimensions~$\neq 2$:

  \[  0 \to \tilde{H}_n(X_i) \to \tilde{H}_n(X_{i+1}) \to 0, \qquad n \geq 4 \]

  \[ \xymatrix{
        0 \ar[r] & \tilde{H}_3(X_i) \ar[r] & \tilde{H}_3(X_{i+1}) \ar[r]
 	         & \tilde{H}_2( \N_{>E_{i+1}} ) \ar[r] & \tilde{H}_2(X_i) \ar[r] & \tilde{H}_2(X_{i+1})\ar[r] & 0 \\
        & &      & \tilde{H}_2 \bigl( C_L(E_{i+1}) \bigr) \ar[u]_{\groupiso}^{i_*} \ar[r] & \tilde{H}_2(\A_2(L)) \ar[u]
               }
   \]

  \[ 0 \to \tilde{H}_1(X_i) \to \tilde{H}_1(X_{i+1}) \to 0 \]

\noindent
Now for degrees~$n \geq 4$:
Exactness in the top line gives an isomorphism in homology for each pair~$(i,i+1)$;
and we compose these over~$i$ to get a homology isomorphism between~$X_0 = \N$ and~$X_s = \N_J$ in those degrees.
We compose these with our earlier equivalences, namely~$\A_2(L) \to \N$ along with~$\N_J \to \A = \A_2(A)$, 
so that we extend to a homology isomorphism between~$\A_2(L)$ and~$\A_2(A)$---establishing conclusion~(3).
Next for degree~$3$:
Exactness at the left of the middle sequence gives monomorphisms~$\tilde{H}_3(X_i) \to \tilde{H}_3(X_{i+1})$,
which we compose over~$i$ to get a monomorphism~$\tilde{H}_3(\N) \to \tilde{H}_3(\N_J)$;
and then composing with the same two earlier homotopy equivalences
gives our desired monomorphism~$\tilde{H}_3 \bigl( \A_2(L) \bigr) \subseteq \tilde{H}_3 \bigl( \A_2(A) \bigr)$,
establishing conclusion~(4).

\bigskip

Finally, we show that with the above computation, we can prove the ``In particular" part of the theorem;
we will apply Corollary~\ref{cor:genTheorem2Segev} with~$\A_2(A)$ in the role of~``$\A$'':
Namely we will show that the inclusion~$\A_2(L) \subseteq \A_2(A)$ is nonzero in homology.
Note that for the even degrees~$2m$ other than~$2$ itself, we get $\tilde{H}_{2m} \bigl( \A_2(L) \bigr) = 0 = \tilde{H}_{2m}(\A)$:
namely for~$2m = 0$ since both posets are connected, and for~$2m \geq 4$ by conclusion~(3) above.
Since homology of odd degree apperas in the Euler characteristic with a negative sign,
while both~$\tilde{\chi}(\A_2(L))$ and $\tilde{\chi}(\A)$ are positive,
we conclude that both the homology groups~$\tilde{H}_2 \bigl( \A_2(L) \bigr)$ and~$\tilde{H}_2(\A)$ are nonzero.
So the epimorphism in degree~$2$ provided by the~$2$-equivalence in conclusion~(2) cannot be the zero map in homology. 
Then the proof finishes with Corollary~\ref{cor:genTheorem2Segev} as promised.
\end{proof}

\section{Elimination of certain alternating components (for the prime 2)}

In this section, we prove part~(3) of Corollary~\ref{corollaryComponents}. 
That is, we will apply the results of earlier sections---to establish~(H-QC)
when~$G$ has an alternating component of type~$\AA_6$ or~$\AA_8$, for~$p = 2$.

After that,
we will discuss how to generalize these results to arbitrary alternating groups $\AA_n$ arising as components of the group $G$.
We focus on the case~$p = 2$, since for odd $p$, $\Out(\AA_n)$ is a~$p'$-group
and hence we can establish~(H-QC) directly via Theorem~\ref{theorem2}.
Finally, we will discuss the case of~(H-QC) for a group~$G$ having exactly two components,
both isomorphic to~$\AA_5$ and interchanged in~$G$:
We show that~(H-QC) holds for this group, first by using our methods;
and then also investigating it via the viewpoint of Theorems~1--3 in~\cite{Segev}.
This comparison shows that the hypotheses of our theorems may be more widely applicable, and easier to check.

\bigskip

We begin below with a direct application of Proposition~\ref{prop:TrivialInclusionCentralizer} to some alternating components.
We refer to~\cite{GLS98} for the assertions on the centralizers in alternating groups.

\begin{corollary}
\label{coroA6}
Assume~$p = 2$.
Suppose that~$O_{2'}(G) = 1$, or that~$G$ satisfies~(H1).
If~$G$ has a component~$L$ isomorphic to~$\AA_6$, then~$G$ satisfies~(H-QC).
\end{corollary}

\begin{proof}
In view of Theorem~\ref{generalReduction}, we can assume that~$O_2(G) = 1 = O_{2'}(G)$.
We check the hypotheses of Proposition~\ref{prop:TrivialInclusionCentralizer}
for the component~$L$ of~$G$ isomorphic~to $\AA_6$.

Although the quotient group~$\Out(\AA_6)$ is elementary of rank~$2$, 
the automorphisms in one of the three nontrivial cosets have preimages in~$\Aut(\AA_6)$ 
which have order~$4$ (the class denoted~$4C$).
So if~$E \in \A_2 \bigl( \Aut(\AA_6) \bigr)$ satisfies~$E \cap \AA_6 = 1$, then~$|E| = 2$;
and hence~$\Outposet_{\Aut(\AA_6)}(\AA_6)$ contains only cyclic~$2$-outers---giving
hypothesis~(1) of Proposition~\ref{prop:TrivialInclusionCentralizer}:
Moreover, for the other outer automorphism groups~$E$ that do have elements of order~$2$ (the classes denoted~$2D$ and~$2B,2C$),
we see that~$C_{\AA_6}(E)$ is isomorphic to either~$D_{10}$ or~$\SS_4$.
In the former case, note~$\A_2(D_{10})$ is discrete with~$5$ points;
and in the latter case~$O_2(\SS_4) > 1$, so~$\A_2(\SS_4)$ is contractible.
To check hypotheses~(2) and~(3) of the Proposition \ref{prop:TrivialInclusionCentralizer}, it remains to find a value~$k \geq 0$,
such that~$\tilde{H}_k \bigl( \A_2(\AA_6) \bigr) \neq 0$,
while~$\A_2 \bigl( C_{\AA_6}(E) \bigr) \hookrightarrow \A_2(\AA_6)$
induces the zero map in the~$k$-th homology group, for all~$E \in I_{\Aut(\AA_6)}(\AA_6)$.
We take~$k := 1$:
For~$\A_2(\AA_6)$ is homotopy equivalent with the Tits building for~$Sp_4(2)$---a wedge of $16$~$1$-spheres,
so that~$\tilde{H}_1 \bigl( \A_2(\AA_6) \bigr) \neq 0$;
while the~$\A_2$-posets of the centralizers are of dimension~$0$ or contractible, 
and hence have zero homology in dimension~$1$---so that the induced maps in $\tilde{H}_1$ are indeed zero, as desired.

Thus we may apply Proposition~\ref{prop:TrivialInclusionCentralizer}, to conclude that~$\tilde{H}_* \bigl( \A_2(G) \bigr) \neq 0$.
\end{proof}

\begin{corollary}
\label{coroA8}
Assume~$p = 2$.
Suppose that $O_{2'}(G) = 1$, or that~$G$ satisfies~(H1).
If~$G$ has a component~$L$ isomorphic to~$\AA_8$, then~$G$ satisfies~(H-QC).
\end{corollary}

\begin{proof}
We proceed in a similar way to the previous proof---namely
we check the hypotheses of Proposition~\ref{prop:TrivialInclusionCentralizer}:
Again we may assume that~$O_2(G) = 1 = O_{2'}(G)$.
This time~$\Out(\AA_8)$ has order~$2$, and so in particular has cyclic~$2$-Sylow subgroups;
so we are in the cyclic-only situation, giving hypothesis~(1).
For~$E \in \Aut(\AA_8)$ giving a~$2$-outer of~$\AA_8$, we have~$C_{\AA_8}(E) \groupiso \SS_6 \groupiso Sp_4(2)$.
Now~$\A_2 \bigl( Sp_4(2) \bigr)$ is in fact homotopy equivalent with~$\A_2(\AA_6)$, namely a wedge of~$16$~$1$-spheres, with homology concentrated in degree~$1$.
To verify hypotheses~(2) and~(3), we will choose the value~$k := 2$:
For~$\A_2(\AA_8)$ is homotopy equivalent to the building for $L_4(2)$, a wedge of~$64$~$2$-spheres, 
so that~$\tilde{H}_2 \bigl( \A_2(\AA_8) \bigr) \neq 0$.
On the other hand, we saw that~$\A_2 \bigl( C_{\AA_8}(E) \bigr)$ has homology concentrated in dimension~$1$,
and hence has zero homology in dimension~$2$,
so that~$\A_2 \bigl( C_{\AA_8}(E) \bigr) \hookrightarrow \A_2(\AA_8)$ induces the zero map in homology, as desired.

Thus we may apply Proposition~\ref{prop:TrivialInclusionCentralizer}, to conclude that~$\tilde{H}_* \bigl( \A_2(G) \bigr) \neq 0$.
\end{proof}

\noindent
We now discuss possible ways of generalizing the above results:

The proofs given above fail for~$\AA_n$ with odd~$n$.
For example, take~$n = 5$:
Then~$\A_2(\AA_5)$ has~$5$ connected components, each contractible; 
so that reduced homology is concentrated in~$\tilde{H}_0 \bigl( \A_2(\AA_5) \bigr)$, a space of dimension~$4$.
By contrast, $\A_2(\SS_5)$ is a connected wedge of~$1$-spheres---so that reduced homology is concentrated in degree~$1$. 
Therefore the inclusion~$\A_2(\AA_5) \hookrightarrow \A_2(\SS_5)$ induces the zero map in homology.
One of the reasons for this behavior in homology is
that the centralizers of outer~$2$-subgroups~$E \leq \SS_5$
have centralizer in~$\AA_5$ equal to~$\SS_3$, which is discrete (so of dimension~$0$) with~$3$ points.
That is, the centralizers have the same maximum nonzero homological dimension as~$\AA_5$.
A similar (but more complex) situation arises with~$\AA_7$ and~$\SS_7$.
The difficulty for odd~$n$ is that
we have a ``leap" of the largest dimension of a nonzero homology group from~$\AA_n$ to~$\SS_n$ (at least for small~$n$).
It would be interesting to study these behaviors for~$n \geq 9$ (both even and odd).

In that direction, we propose the problem below.
If~$r \geq 0$ is a real number, denote by~$[r]$ the largest integer~$n$ with~$n\leq r$ (the floor function). 

\vspace{0.3cm}

\textbf{Problem.}
Let~$m_a := [n/2]-2$ for~$n \geq 5$, and~$m_s := [(n+1)/2]-2$ for~$n \neq 2,4$.
Show that~$m_a$ (resp. $m_s$)~is the largest integer
such that~$\A_2(\AA_n)$~(resp. $\A_2(\SS_n)$) has nonzero homology in degree~$m_a$ (resp.~$m_s$).

\vspace{0.3cm}

Suppose that the statements in the above Problem  have been established.
Since~$C_{\AA_n}(E)\groupiso \SS_{n-2}$ is the centralizer of a particular outer involution of~$\AA_n$,
we see that~$\A_2(C_{\AA_n}(E))$ and~$\A_2(\AA_n)$ share the same largest dimension for a nonzero homology group
if and only if~$m_s(n-2) = m_a(n)$: 
   \[ [((n-2)+1)/2] - 2 = [n/2]-2 , \]
that is, if and only if
   \[ [(n-1)/2] = [n/2] . \]
This equality holds if and only if~$n$ is odd.
Therefore, modulo the above Problem, for even~$n$ we can proceed as in Corollary \ref{coroA8},
and deduce that~(H-QC) holds for~$G$, if it satisfies~(H1) and has a component of type~$\AA_{n}$.

\bigskip

We conclude this section with an application of Theorem~\ref{theoremNormalCaseReduction},
to establish~(H-QC) for certain groups with~$\AA_5$-components.
We will establish Property~E(2), 
and use that to directly verify the hypothesis of Theorem~\ref{theoremNormalCaseReduction}---rather than
using the inductive procedure of Proposition~\ref{propositionPropertiesEM}. 
It would be interesting to investigate if this kind of proof can show~(H-QC) in groups with~$\AA_n$ components
for larger odd~$n$.

\begin{example}
Take~$p := 2$.
Let~$G := (\AA_5\times\AA_5) \rtimes (E \times R)$,
where~$E \groupiso C_2$ acts diagonally by outer involutions on both copies of~$\AA_5$,
and~$R \groupiso C_2$ interchanges the components.

Thus our~$G$-orbit has length~$t = 2$.
Note that the kernel on the components is~$H = (\AA_5\times\AA_5) \rtimes E$.
We will apply Theorem~\ref{theoremNormalCaseReduction} to show~(H-QC) for~$G$ for~$p = 2$; 
that is, we will show that~$\psi_H$ is not the zero map in homology.
We will accomplish this by first establishing Property~E(2) for the~$\varphi$-maps corresponding to the components.

\bigskip

Let~$L_1$ and~$L_2$ be the copies of~$\AA_5$ which are components of~$G$, and  set~$N := L_1 L_2 = F^*(G)$.
Note that~$G$ interchanges~$L_1$ and~$L_2$, and that~$O_2(G) = 1 = O_{2'}(G)$.
With the notation of Definition~\ref{defn:CalAi}, we have:
  \[ C_0(H) = C_H(L_1 L_2) = 1 , \quad \quad C_1(H) = C_H(L_2) = L_1 , \quad \quad C_2(H) = H .\]
As promised above, we will first establish Property~E(2) for the pairs~$(L_i, C_i(H))$;
in fact we will show that~$\varphi_i$ induces an epimorphism in {\em all\/} homology groups,
for the relevant values~$1 \leq i \leq t = 2$.

\bigskip

\noindent
For~$i = 1$:
Using~$C_0(H) = 1$ and~$C_1(H) = L_1$ above, we see that~$\varphi_1$ maps~$\A_2(L_1)$ to:

\centerline{
   $\A_2(1) * \A_1 = \emptyset * \A_{L_1,L_1} = \A_2(L_1)$.
            }

\noindent
Thus~$\varphi_1$ is the identity map, and in particular induces an epimorphism in all homology groups.

\bigskip

\noindent
For~$i = 2$, the proof will be much lengthier:

This time using~$C_1(H) = L_1$ and~$C_2(H) = H$,
we see that~$\varphi_2$ maps~$\A_2(H)$ to~$\A_2(L_1) * \A_2$, where~$\A_2 = \A_{H,L_2} \cong \A_2(\SS_5)$.
In view of~Lemma~\ref{lemmaPOuterPreimage}, since we have~$C_H(L_2) = L_1$,
we can regard~$\A_2$ as the set~$\{ L_1 A \tq A \in \A_2(H),\ C_A(L_2) = 1 \}$; 
where~$L_1 A$ is a coset of~$L_1$ in the subset~$\A_2$ of the quotient~$H/C_H(L_2) = H/L_1$. 
We want~$\varphi_2$ to induce an epimorphism in homology.

In fact, we will prove that~$\varphi_2$ is even a~$2$-equivalence:
namely an epimorphism in degree~$2$, but a homology isomorphism in smaller degrees.
For this, we will apply the variant of Quillen's fiber lemma which we gave as Proposition~\ref{variantQuillenFiber}.
Set:

\centerline{
    $Y := \A_2(L_1) * \A_2$, 
            }

\noindent
and note that the join~$Y$ is~$2$-dimensional.
In view of this proposition, it is enough to show, for all~$y \in Y$,
that~$W_y := \varphi_2^{-1}( Y_{\leq y} ) * Y_{>y}$ is~$1$-connected.

Assume first the case where~$y \in \A_2(L_1)$.
Here we see~$\varphi_2^{-1}( Y_{\leq y} ) = \A_2(L_1)_{\leq y}$ is contractible;
so that~$W_y$ is also contractible, and in particular is~$1$-connected.

So we may assume instead the other case, where~$y \in \A_2$.
We saw above that we may regard~$y$ as~$L_1 F_0$, for some~$F_0 \in \A_2(H)$ such that~$C_{F_0}(L_2) = F_0 \cap L_1 = 1$
(since $C_H(L_2) = L_1$).
Therefore:
  \[ \varphi_2^{-1}( Y_{\leq y} ) = \A_2(L_1)\ \cup\ \{\ F \in \A_2(H) - \A_2(L_1) \tq L_1F \leq L_1 F_0\ \}
                                  = \A_2(L_1 F_0) .
    \]
Assume first the case where~$F_{01} := F_0 \cap L_1 C_H(L_1) = F_0 \cap L_1 L_2 > 1$. 
Then~$L_1 F_0$ contains the projections~$F_i$ of~$F_{01}$ on the~$L_i$---with~$F_2 > 1$ since~$F_0 \cap L_1 = 1$. 
We see that~$F_2 \leq Z(L_1 F_0)$, so that~$1 < F_2 \leq O_2(L_1 F_0)$;
thus~$\A_2(L_0 F_1)$ is contractible, and hence so is~$W_y$. 

Hence we may assume the remaining case where~$F_0 \cap L_1 C_H(L_1) = 1$---so
that~$F_0$ acts faithfully on~$L_1$ by outer automorphisms.
Then~$F_0$ has order~$2$, and~$\A_2(L_1 F_0) = \A_2(\SS_5)$ is~$0$-connected (i.e.~connected).%
\footnote{
  Notice this is our first use of the specific~$\AA_5$-structure in our hypothesis.
          }
On the other hand, $Y_{>y} = (\A_2)_{>y}$.
Now since $H/C_H(L_2) = H/L_1 \cong L_2 E \cong \SS_5$, 
we can regard~$\A_2 = \A_{H,L_2}$ in the alternative form of~$\A_2(L_2 E) = \A_2(\SS_5)$;
and in this representation, our element~$y \in \A_2$
is an elementary abelian~$\tilde{F}_0$ of order~$2$, inducing outer automorphisms on~$L_2$.
Hence:
  $$ \A_2(\SS_5)_{>y} = \N_{\SS_5}(L_2)_{>\tilde{F}_0} \simeq \A_2 \bigl( C_{L_2}(\tilde{F}_0) \bigr) = \A_2(\SS_3) , $$
which is~$(-1)$-connected (i.e.~nonempty).
Therefore, we have that~$W_y \simeq \A_2(\SS_5) * \A_2(\SS_3)$ is in fact~$1$-connected,
by the classical join property mentioned after~(\ref{equationJoinHomology})---with~$0,-1$ in the roles of~``$n,m$''.
(Moreover, $W_y$ is homotopy equivalent
to a bouquet of:
  \[ \dim H_1 \bigl( \A_2(\SS_5) \bigr)\ \cdot\ \dim \tilde{H}_0 \bigl( \A_2(\SS_3) \bigr) = 16 \cdot 2 = 32 \]
spheres of dimension~$2$.)

\bigskip
   
We have now shown, for all~$y \in Y = \A_2(L_1) * \A_2$,
that~$W_y$ is~$1$-connected (and indeed often contractible). 
By Proposition~\ref{variantQuillenFiber}, we see that~$\varphi_2$ is a~$2$-equivalence;
and in particular, it is an epimorphism in all homology groups of degree at most~$2$.
Since~$Y$ has dimension~$2$, its homology groups vanish in degree~$>2$;
so in fact~$\varphi_2$ is an epimorphism in {\em all\/} the homology groups.
This fact for~$\varphi_2$ in particular completes the proof of Property~E(2) for all (i.e.~both)~$\varphi_i$. 

Now to complete the hypothesis of Theorem~\ref{theoremNormalCaseReduction},
we want~$\tilde{H}_2(\ \A_2 \bigl( C_H(N) \bigr) * \A_1 * \A_2\ ) \neq 0$.
In this case, $C_H(N) = 1$ since there are no other components outside our orbit;
while~$\A_1 = \A_2(L_1)$ and~$\A_2 = \A_2(\SS_5)$.
Hence using~(\ref{equationJoinHomology}):

\centerline{
$  \tilde{H}_2(\ \A_2 \bigl( C_H(N) \bigr) * \A_1* \A_2\ )
 = \tilde{H}_2 \bigl(\ \A_2(\AA_5)  * \A_2(\SS_5)\ \bigl) 
 = \tilde{H}_0 \bigl( \A_2(\AA_5) \bigr) \otimes \tilde{H}_1 \bigl( \A_2(\SS_5) \bigr)
 \neq 0$ ;
            }

\noindent
for~$\neq 0$ follows by the almost-simple case in Theorem~\ref{theoremAK},
while the more detailed expression follows by direct computation---since $\A_2(\AA_5)$ is a wedge of~$4$~$0$-spheres,
and~$\A_2(\SS_5)$ is a wedge of~$16$~$1$-spheres.
Hence applying the Theorem, (H-QC) also holds for~$G$.
(Along the way, the computation has shown that~(H-QC) holds for~$H$ as well.)

\bigskip

\noindent
Finally, we approach the proof using the viewpoint of Segev's theorems: 

First note that Segev's Theorems~2 and~3 (which we stated as Theorem~\ref{theoremSegevOriginal}) cannot be applied here:
We have that~$H$ is the kernel on components, and~$H > F^*(H) = \AA_5\times\AA_5$---so that we can't use Theorem~3.
Moreover, if~$L$ is a component of~$H$, then~$\Aut_H(L) = \SS_5$;
so~$\A_2(L) \to \A_2 \bigl( \Aut_H(L) \bigr)$ is the zero map in homology:
For~$\A_2(L)$ has~$5$ connected components, each contractible, 
and hence has reduced homology given by~$4$-dimensional~$\tilde{H}_0$;
while~$\A_2 \bigl( \Aut_H(L) \bigr) = \A_2(\SS_5)$ is connected, and so has zero in~$\tilde{H}_0$. 
Hence we don't get the nonzero map required for the hypothesis of Theorem~2.

Nevertheless, it is possible to get~(H-QC) here, by applying Theorem~1 of~\cite{Segev}:
Note first that~$\D_2(H) = \A_2(H) - \bigl( \A_2(L_1) \cup \A_2(L_2) \bigr)$.
And we can conclude, using computations in GAP, that $\D_2(H) \to \A_2(H)$ is not surjective in homology:
since~$H_2 \bigl( \A_2(H) \bigr)$ has dimension~$384$, while~$H_2 \bigl( \D_2(H) \bigr)$ has dimension only~$36$.
Hence~$\tilde{H}_* \bigl (\A_2(G) \bigr) \neq 0$ by Theorem~1 of~\cite{Segev}.
\donerk
\end{example}

\end{document}